\documentclass[11pt]{article}

\usepackage{fancyhdr}
\usepackage{amsfonts}
\usepackage{amsmath}
\usepackage{amssymb}
\usepackage{amsthm}
\usepackage{tikz}

\usepackage{amscd}
\usepackage{amstext}

\usepackage{oldgerm} 

\usepackage{fancyhdr}
\usepackage{amsfonts}
\usepackage{amsmath}
\usepackage{amssymb}
\usepackage{amsthm}

\usepackage{amscd}
\usepackage{amstext}

\usepackage{graphics}
\usepackage{graphicx}

\newlength{\fixboxwidth}
\newcommand{\re}{\mathbb{R}}\newcommand{\N}{\mathbb{N}}
\newcommand{\zz}{\mathbb{Z}}

\newcommand{\com}{\mathbb{C}}

\newcommand{\R}{{\re}^d}

\newcommand{\cs}{{\mathcal S}}

\newcommand{\cl}{{\mathcal L}}

\newcommand{\cf}{{\mathcal F}}
\newcommand{\cfi}{{\cf}^{-1}}

\newcommand{\supp}{{\rm supp \, }}

\newcommand{\dist}{{\rm dist \, }}

\newcommand{\be}{\begin{equation}}
\newcommand{\ee}{\end{equation}}
\newcommand{\beq}{\begin{eqnarray}}
\newcommand{\beqq}{\begin{eqnarray*}}
\newcommand{\eeq}{\end{eqnarray}}
\newcommand{\eeqq}{\end{eqnarray*}}

\newtheorem{satz}{Theorem}

\newtheorem{rem}{Remark}
\newtheorem{defi}{Definition}
\newtheorem{lem}{Lemma}
 
\newtheorem{prop}{Proposition}

\begin{document}


\title{Besov-Morrey spaces and differences (extended version)}
\author{Marc Hovemann}
\date{\today}
\maketitle

\begin{center}
{\scriptsize  Institute of Mathematics, Friedrich-Schiller-University Jena, Ernst-Abbe-Platz 2, 07743 Jena, Germany }
\end{center}

\vspace{0.3 cm }


\vspace{0.3 cm }


\vspace{0.3 cm }

\textbf{Abstract.} We study the Besov-Morrey spaces $ \mathcal{N}^{s}_{u,p,q}(\mathbb{R}^{d}) $ and show that under certain conditions on the parameters these spaces can be characterized in terms of higher-order differences. Furthermore we prove that some of the mentioned conditions are also necessary.


\section{Introduction and main results}


Nowadays the Besov spaces $B^s_{p,q} (\R)$ are a well-established tool to describe the regularity of functions and distributions. These function spaces have been introduced by Nikol'skij and Besov between 1951 and 1961, see \cite{Ni2}, \cite{Be1959} and \cite{Be1961}. Later the spaces $  B^s_{p,q} (\R) $ have been investigated in detail in the famous books of Triebel, see \cite{Tr83}, \cite{Tr92} and \cite{Tr06}. In the recent years a growing number of authors worked with a generalization of the Besov spaces where the $ L_{p} $ quasi-norm was replaced by a Morrey quasi-norm. These function spaces are called Besov-Morrey spaces and have the symbol $  \mathcal{N}^{s}_{u,p,q}(\R) $ with $ 0 < p \leq u < \infty$, $ 0 < q \leq \infty $ and $ s \in \mathbb{R} $. The Besov-Morrey spaces have been introduced by Kozono and Yamazaki in 1994, see \cite{KoYa}. Later they were studied by Mazzucato, see \cite{Maz2003}. Here the spaces $  \mathcal{N}^{s}_{u,p,q}(\R) $ appeared in connection with Navier-Stokes equations. More information concerning the Besov-Morrey spaces and other smoothness Morrey spaces that are related to them can be found in \cite{ysy}. This paper has two main goals. The first one is to prove an equivalent characterization in terms of higher-order differences for the spaces $ \mathcal{N}^{s}_{u,p,q}(\R) $. More exactly we will answer the question under which restrictions on the parameters $s,u,p,q$ and $d$ the spaces $\mathcal{N}^{s}_{u,p,q}(\R) $ can be described by using only $\Delta_h^N f(x)$. When we solve this problem we will obtain some sufficient conditions concerning the parameter $ s $. Because of this our second main goal is to investigate whether these conditions are also necessary. For the original Besov spaces $B^s_{p,q} (\R)$ characterizations in terms of differences are known since many years. So for example the following result can be found in a famous book of Triebel from 1983, see chapter 2.5.12. in \cite{Tr83}.

\begin{satz}\label{B_hist_dif}
Let $ 0 < p \leq \infty $, $ 0 < q \leq \infty $  and $ s > d \max  ( 0, \frac{1}{p} - 1  ) $. Let $ N \in \mathbb{N} $ such that $ N > s $. Then 
\begin{align*} 
\Vert f \vert L_{p}(\R) \Vert + \Big ( \int_{\R} \vert h \vert^{-sq} \Vert    \Delta^{N}_{h}f(x)   \vert L_{p}(\R) \Vert^{q} \frac{dh}{\vert h \vert^{d}} \Big )^{\frac{1}{q}} 
\end{align*}
is an equivalent quasi-norm in $B^s_{p,q} (\R)$. In the case $ q = \infty $ the usual modifications have to be made. 
\end{satz}

In this paper we want to prove related results for the more general Besov-Morrey spaces. The following theorem is one of our main results and shows how the spaces $ \mathcal{N}^{s}_{u,p,q}(\R)  $ can be characterized in terms of differences.

\begin{satz}\label{MR1}
Let $ 0 < p \leq u < \infty $ and $ 0 < q \leq \infty $. Let $ 0 < v \leq \infty $ and 
\begin{align*}
 s ~ > ~ d \max \Big  ( 0, \frac{1}{p} - 1, \frac{1}{p} - \frac{1}{v}   \Big ). 
\end{align*}
Let $ N \in \mathbb{N} $ with $ N > s $. Then a function $ f~ \in ~ L_{p}^{loc}(\R)$ belongs to $ \mathcal{N}^{s}_{u,p,q}(\R)   $ if and only if $ f \in L_{v}^{loc}(\R)$ and (modifications if $ q = \infty $ and/or $ v = \infty $) 
\begin{align*}
& \underbrace{\Vert f \vert \mathcal{M}^{u}_{p}( \mathbb{R}^d)   \Vert + \Big ( \int_{0}^{ \infty } t^{-sq-d \frac{q}{v}} \Big \Vert  \Big ( \int_{B(0,t)}\vert \Delta^{N}_{h}f(x) \vert^{v} dh \Big )^{\frac{1}{v}} \Big \vert  \mathcal{M}^{u}_{p}( \mathbb{R}^d) \Big \Vert^{q}  \frac{dt}{t} \Big )^{\frac{1}{q}}  } < \infty.  \\
& \hspace{4,5 cm} \Vert f \vert \mathcal{N}^{s}_{u,p,q}(\R) \Vert^{(v, \infty)} :=  
\end{align*}
The quasi-norms $ \Vert f \vert \mathcal{N}^{s}_{u,p,q}(\R) \Vert  $ and $ \Vert f \vert \mathcal{N}^{s}_{u,p,q}(\R) \Vert^{(v, \infty )} $  are equivalent for $ f \in L_{p}^{loc}(\R)$.
\end{satz}  
As already mentioned the second main goal of this paper is to investigate whether the conditions concerning the parameter $ s $ that you can find in theorem \ref{MR1} are also necessary. For that purpose by $ {\bf N}^{s, N, \infty }_{u, p, q, v}(\R) $ we define the collection of all $ f \in L_{\max(p,v)}^{loc}(\R)$ such that 
$ \Vert f \vert \mathcal{N}^{s}_{u,p,q}(\R)  \Vert^{(v, \infty )} $ is finite. Using this notation we can formulate the following theorem which is our second main result.  

\begin{satz}\label{MR2}

Let $ s \in \mathbb{R}  $, $ 0 < p \leq u < \infty $, $ 0 < q \leq \infty $, $ 0 < v \leq \infty $ and $ N \in \mathbb{N} $. Then we have $  \mathcal{N}^{s}_{u,p,q}(\R) \not =  {\bf N}^{s, N, \infty}_{u, p, q, v}(\R)    $ if we are in one of the following cases. 

\begin{itemize}
\item[(i)] We have $ s \leq 0  $.

\item[(ii)] We have $ 0 < p < 1 $ and
\begin{itemize}
\item[(a)]
either $ s < d \frac{p}{u} \left ( \frac{1}{p} - 1   \right ) $
\item[(b)]
or $ s = d \frac{p}{u} \left ( \frac{1}{p} - 1   \right ) $ and $ q > 1 $.
\end{itemize}

\item[(iii)] We have $ s < d  \frac{p}{u} \left ( \frac{1}{p} - \frac{1}{v}  \right )   $ with $ 0 < p < v < \infty  $.

\item[(iv)] We have 
\begin{itemize}
\item[(a)]
either $ N < s $ and $ 0 < q \leq \infty $
\item[(b)]
or $ N = s $ and $ 0 < q < \infty $
\item[(c)]
or $ N = s $ with $ q = \infty  $ and $ u = p  $ and $  v \geq 1 $.
\end{itemize}
\end{itemize}
\end{satz}
If you compare this result with theorem \ref{MR1} it turns out that there are still some open questions at the moment. Let us look at the special case $ v = 1 $ and $ 0 < p < 1 $. Then for $  d  \frac{p}{u}  ( \frac{1}{p} - 1   ) < s \leq d  ( \frac{1}{p} - 1   )   $ it is not clear whether it is possible to describe the spaces $ \mathcal{N}^{s}_{u,p,q}(\R)  $ in terms of differences or not. Notice that this gap disappears in the case of the original Besov spaces $ B^{s}_{p,q}(\R) $. It is possible to illustrate our results with a $  ( \frac{1}{p} , s  ) $ - diagram. For convenience in the following diagram we assume $ u = 1 $ if $ 0 < p < 1 $. The influence of the parameter $ q $ is hidden. In the area A we have 
$ \mathcal{N}^{s}_{u,p,q}(\R)  =  {\bf N}^{s, N, \infty}_{u, p, q, 1}(\R) $ and in B we find $ \mathcal{N}^{s}_{u,p,q}(\R) \not =  {\bf N}^{s, N, \infty}_{u, p, q, 1}(\R) $. In area C at the moment it is not clear whether we have $ \mathcal{N}^{s}_{u,p,q}(\R)  =  {\bf N}^{s, N, \infty}_{u, p, q, 1}(\R) $ or $ \mathcal{N}^{s}_{u,p,q}(\R) \not =  {\bf N}^{s, N, \infty}_{u, p, q, 1}(\R) $.


\begin{minipage}[b]{0.2\textwidth}
\begin{picture}(0,0)%
\includegraphics{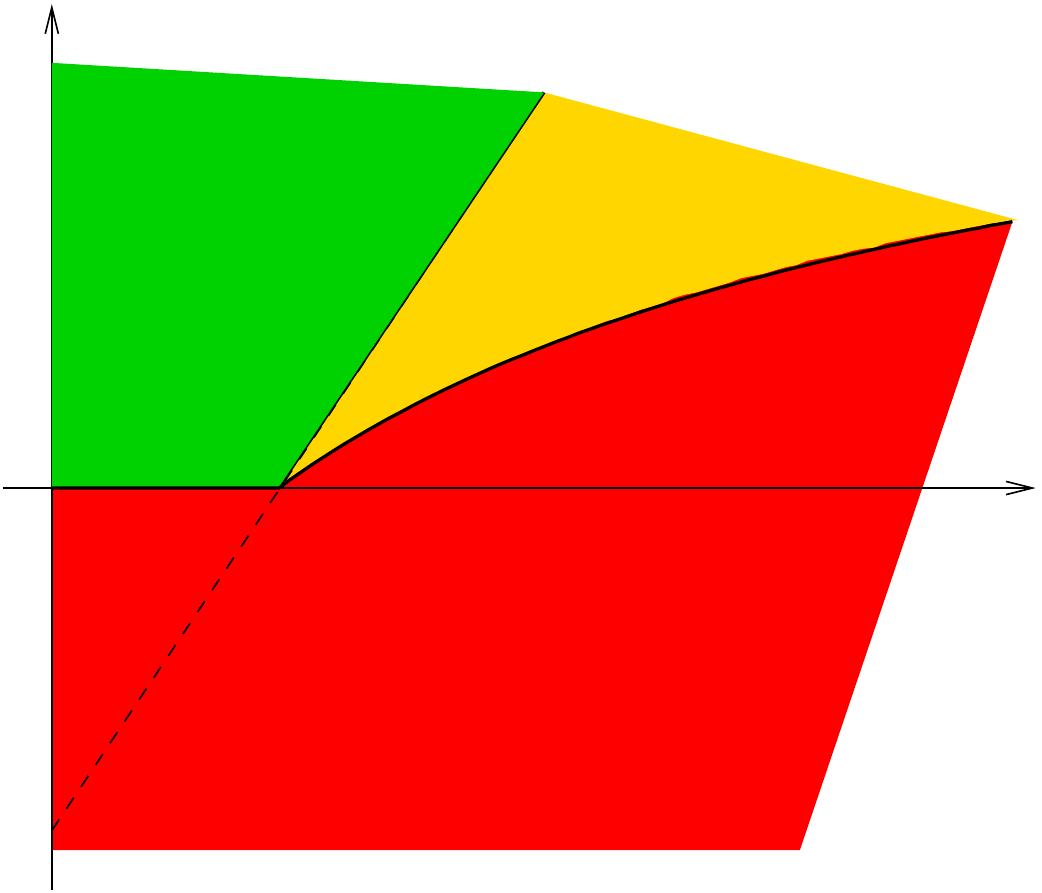}%
\end{picture}%
\setlength{\unitlength}{4144sp}%
\begingroup\makeatletter\ifx\SetFigFont\undefined%
\gdef\SetFigFont#1#2#3#4#5{%
  \reset@font\fontsize{#1}{#2pt}%
  \fontfamily{#3}\fontseries{#4}\fontshape{#5}%
  \selectfont}%
\fi\endgroup%
\begin{picture}(4749,4074)(2014,-4123)
\put(2836,-1231){\makebox(0,0)[b]{\smash{{\SetFigFont{10}{12.0}{\familydefault}{\mddefault}{\updefault}{\color[rgb]{0,0,0}$A$}%
}}}}
\put(4726,-1096){\makebox(0,0)[b]{\smash{{\SetFigFont{10}{12.0}{\familydefault}{\mddefault}{\updefault}{\color[rgb]{0,0,0}$C$}%
}}}}
\put(4141,-3121){\makebox(0,0)[b]{\smash{{\SetFigFont{10}{12.0}{\familydefault}{\mddefault}{\updefault}{\color[rgb]{0,0,0}$B$}%
}}}}
\put(6571,-2491){\makebox(0,0)[b]{\smash{{\SetFigFont{10}{12.0}{\familydefault}{\mddefault}{\updefault}{\color[rgb]{0,0,0}$\frac1p$}%
}}}}
\put(2206,-2491){\makebox(0,0)[rb]{\smash{{\SetFigFont{10}{12.0}{\familydefault}{\mddefault}{\updefault}{\color[rgb]{0,0,0}$0$}%
}}}}
\put(3286,-2491){\makebox(0,0)[lb]{\smash{{\SetFigFont{10}{12.0}{\familydefault}{\mddefault}{\updefault}{\color[rgb]{0,0,0}$1$}%
}}}}
\put(2206,-331){\makebox(0,0)[rb]{\smash{{\SetFigFont{10}{12.0}{\familydefault}{\mddefault}{\updefault}{\color[rgb]{0,0,0}$s$}%
}}}}
\put(2206,-3886){\makebox(0,0)[rb]{\smash{{\SetFigFont{10}{12.0}{\familydefault}{\mddefault}{\updefault}{\color[rgb]{0,0,0}$-d$}%
}}}}
\put(4186,-871){\makebox(0,0)[rb]{\smash{{\SetFigFont{10}{12.0}{\familydefault}{\mddefault}{\updefault}{\color[rgb]{0,0,0}$s=\sigma_p$}%
}}}}
\put(4366,-1816){\makebox(0,0)[lb]{\smash{{\SetFigFont{10}{12.0}{\familydefault}{\mddefault}{\updefault}{\color[rgb]{0,0,0}$s=\max(0,\frac{d}{u}(1-p))$}%
}}}}
\end{picture}%
\end{minipage}\hfill
\begin{minipage}[b]{0.29\textwidth}

\textbf{Figure 1.} Characterizations in terms of differences for $ \mathcal{N}^{s}_{u,p,q}(\R) $. 

\vspace{1 cm}

\end{minipage}


This paper is organized in the following way. In chapter 2 we want to define the Besov-Morrey spaces $ \mathcal{N}^{s}_{u,p,q}(\R) $ precisely. In addition we want to collect some useful properties concerning these spaces. In chapter 3 of this article we want to prove main theorem \ref{MR1}. To do so we use some methods developed by Hedberg and Netrusov, see \cite{HN}. In chapter 4 of this paper we want to prove the results concerning necessity including main theorem \ref{MR2}. But first of all we want to fix some notation.

\section*{Notation}

As usual $\N$ denotes the natural numbers, $\N_0$ the natural numbers including $0$, $\zz$ the integers and $\re$ the real numbers. $\R$ denotes the $d$-dimensional  Euclidean space. We put
\[
 B(x,t) := \{y\in \R: \quad |x-y|< t\}\, , \qquad x \in \R\, , \quad t>0.
\]
All functions are assumed to be complex-valued, i.\,e. we consider functions $f:~ \R \to \com$. Let $\mathcal{S}(\R)$ be the collection of all Schwartz functions on $\R$ endowed with the usual topology and denote by $\mathcal{S}'(\R)$ its topological dual, namely the space of all bounded linear functionals on $\mathcal{S}(\R)$ endowed with the weak $\ast$-topology. The symbol $\cf$ refers to  the Fourier transform,
$\cfi$ to its inverse transform, both defined on $\cs'(\R)$. Almost all function spaces which we consider in this paper are subspaces of $\cs'(\R)$, i.\,e. spaces of equivalence classes with respect to almost everywhere equality. However, if such an equivalence class contains a continuous representative, then usually we work with this representative and call also the equivalence class a continuous function. 
By $C^\infty_0(\R)$ we mean the set of all infinitely often differentiable functions on $\R$ with compact support. Given a quasi-Banach space $X$ the operator norm of a linear operator $T:\, X\to X$ is denoted by $\|T|\cl (X)\|$. For two quasi-Banach spaces $ X $ and $ Y $ we write $ X \hookrightarrow Y $ if $ X \subset Y $ and the natural embedding of $ X $ into $ Y $ is continuous. For all $p\in(0,\infty )$ and $ q \in (0 , \infty ] $ we write 
\[
\sigma_p:= d\,  \max \Big(0, \frac 1p - 1\Big) \qquad \mbox{and}\qquad 
 \sigma_{p,q}:= d\,  \max \Big(0, \frac 1p -1 , \frac 1q - 1 \Big) \, .
\]
The symbols  $C, C_1, c, c_{1} \ldots $ denote positive constants that depend only on the fixed parameters $d,s,u,p,q$ and probably on auxiliary functions. Unless otherwise stated their values may vary from line to line. In this paper one important tool will be differences of higher order. Let $ f : \mathbb{R}^d \rightarrow \mathbb{C}$ be a function. Then for $ x, h \in \mathbb{R}^d $ we define the difference of the first order by $ \Delta_{h}^{1}f (x) := f ( x + h ) - f (x) $. Let $ N \in \mathbb{N}$. Then we define the difference of order $ N $ by 
\[
\Delta_{h}^{N}f (x) := \left  (\Delta_{h} ^1 \left ( \Delta_{h} ^{N-1}f \right  )\right ) (x) \, , \qquad x \in \R\, . 
\]

\section{Definition and basic properties of Besov-Morrey spaces}

The Besov-Morrey spaces $ \mathcal{N}^{s}_{u,p,q}(\R)   $ are function spaces that are built upon Morrey spaces. Because of this at first we want to recall the definition of the Morrey spaces $ \mathcal{M}^{u}_{p}(\R)  $. 

\begin{defi}\label{def_mor}

Let $ 0 < p \leq u < \infty$. Then the  Morrey space $ \mathcal{M}^{u}_{p}(\R)  $ is defined to be the set of all functions $ f \in L_{p}^{loc}(\R) $ such that 
\begin{align*}
\Vert f \vert \mathcal{M}^{u}_{p}(\R) \Vert := \sup_{y \in \R, r > 0} \vert B(y,r) \vert^{\frac{1}{u}-\frac{1}{p}} \Big ( \int_{B(y,r)} \vert f(x) \vert^{p} dx      \Big )^{\frac{1}{p}} < \infty.
\end{align*} 

\end{defi}

The Morrey spaces $ \mathcal{M}^{u}_{p}(\R)  $ are quasi-Banach spaces and Banach spaces for $ p \geq 1$. They have many connections to the Lebesgue spaces $ L_{p}(\R)$. So for $ p \in (0,\infty) $ we have $ \mathcal{M}^{p}_{p}(\R) = L_{p}(\R)$. Moreover for $ 0 < p_{2} \leq p_{1} \leq u < \infty $ we have

\begin{center}
$ L_{u}(\R) = \mathcal{M}^{u}_{u}(\R) \hookrightarrow   \mathcal{M}^{u}_{p_{1}}(\R)  \hookrightarrow  \mathcal{M}^{u}_{p_{2}}(\R)  $.
\end{center} 

To define the spaces $ \mathcal{N}^{s}_{u,p,q}(\R)   $ we need a so-called smooth dyadic decomposition of the unity. Let $\varphi_0 \in C_0^{\infty}({\R})$ be a non-negative function such that $\varphi_0(x) = 1$ if $|x|\leq 1$ and $ \varphi_0 (x) = 0$ if $|x|\geq \frac{3}{2}$. 
For $k\in \N$ we define $  \varphi_k(x) := \varphi_0(2^{-k}x)-\varphi_0(2^{-k+1}x)  $. Because of $  \sum_{k=0}^\infty \varphi_k(x) = 1 $ and $ \supp \varphi_k \subset \big\{x\in \R: \: 2^{k-1}\le |x|\le 3 \cdot 2^{k-1}\big\}  $ for every $ k \in \mathbb{N}  $ we  call the system $(\varphi_k)_{k\in \N_0 }$ a smooth dyadic decomposition of the unity on $\R$. For $k \in \N_0$ because of the Paley-Wiener-Schwarz theorem $\cfi[\varphi_{k}\, \cf f]$ is a smooth function for all $f\in \cs'(\R)$. If we use the system $(\varphi_k)_{k\in \N_0 }$ we are able to define the Besov-Morrey spaces $ \mathcal{N}^{s}_{u,p,q}(\R) $.

\begin{defi}\label{def_bms}

Let $ 0 < p \leq u < \infty $, $ 0 < q \leq \infty $ and $ s \in \mathbb{R} $. $ (\varphi_{k})_{k\in \N_0 }$ is a smooth dyadic decomposition of the unity. Then the Besov-Morrey space $  \mathcal{N}^{s}_{u,p,q}(\mathbb{R}^{d}) $ is defined to be the set of all distributions $ f \in \mathcal{S}'(\mathbb{R}^{d})  $ such that
\begin{align*} 
\Vert f \vert \mathcal{N}^{s}_{u,p,q}(\mathbb{R}^{d})  \Vert :=  \Big ( \sum_{k = 0}^{\infty} 2^{ksq}   \Vert \mathcal{F}^{-1}[\varphi_{k} \mathcal{F}f]  \vert \mathcal{M}^{u}_{p}(\R)   \Vert  ^{q} \Big   )^{\frac{1}{q}} < \infty .
\end{align*}

In the case $ q = \infty $ the usual modifications are made.

\end{defi}

In what follows we want to collect some basic properties of the Besov-Morrey spaces. Most of them will be used later.

\begin{lem}\label{l_bp1}
Let $ 0 < p \leq u < \infty $, $ 0 < q \leq \infty $ and $ s \in \mathbb{R} $. Then the following assertions are true.

\begin{itemize}
\item[(i)]

The spaces $  \mathcal{N}^{s}_{u,p,q}(\mathbb{R}^{d}) $ are independent of the chosen smooth dyadic decomposition of the unity in the sense of equivalent quasi-norms. 

\item[(ii)]

The spaces $  \mathcal{N}^{s}_{u,p,q}(\mathbb{R}^{d}) $ are quasi-Banach spaces. For $ p \geq 1 $ and $ q \geq 1 $ they are Banach spaces.

\item[(iii)]

Let $ \tau = \min(1,p,q) $. Then we have

$  \Vert f + g \vert \mathcal{N}^{s}_{u,p,q}(\mathbb{R}^{d})  \Vert^{\tau} \leq \Vert f \vert \mathcal{N}^{s}_{u,p,q}(\mathbb{R}^{d})  \Vert^{\tau} + \Vert g \vert \mathcal{N}^{s}_{u,p,q}(\mathbb{R}^{d})  \Vert^{\tau} $ for all $ f,g \in \mathcal{N}^{s}_{u,p,q}(\mathbb{R}^{d})  $.

\item[(iv)]

It holds $\mathcal{S}(\mathbb{R}^{d}) \hookrightarrow    \mathcal{N}^{s}_{u,p,q}(\mathbb{R}^{d}) \hookrightarrow   \mathcal{S}'(\mathbb{R}^{d})$.

\item[(v)]

We have $ \mathcal{N}^{s}_{p,p,q}(\mathbb{R}^{d}) = B^{s}_{p,q}(\R)  $.

\end{itemize}

\end{lem} 

\begin{proof}

(i) was proved in \cite{TangXu}, see theorem 2.8. The proofs of (ii) and (iii) are standard, see corollary 2.6. in \cite{KoYa}. (iv) was proved in \cite{SawTan}, see theorem 3.2.  (v) is obvious, see proposition 3.6. in \cite{SawTan}. 
\end{proof}

The Besov-Morrey spaces have many connections to other function spaces. For example they are embedded into some Besov or Triebel-Lizorkin-Morrey spaces. Results concerning this topic can be found in \cite{HaSkCE} or in \cite{HaSk}. For us it is interesting to know that in some cases the spaces $  \mathcal{N}^{s}_{u,p,q}(\mathbb{R}^{d}) $ also contain singular distributions. The following result is a combination of theorem 3.3 and theorem 3.4. from \cite{HaMoSk} and theorem 3.4. from \cite{HaMoSkarXiv}.

\begin{lem}\label{l_bp2}

Let $ s \in \mathbb{R}  $,  $ 0 < p \leq u < \infty $ and $ 0 < q \leq \infty$. Then we have $ \qquad \qquad $ $\mathcal{N}^{s}_{u,p,q}(\R) \not \subset L_{1}^{loc}(\R) $ if and only if we are in one of the following cases.

\begin{itemize}
\item[(i)]

We have $ s < \frac{p}{u} \sigma_p  $.

\item[(ii)]

We have $ s = \frac{p}{u} \sigma_p $ with $ 0 < p < 1 $ and $ q > 1 $.

\item[(iii)]

We have $ s = 0 $ with $ p \geq 2 $ and $ q > 2 $.

\item[(iv)]

We have $ s = 0 $ with $ 1 \leq p < 2 $ and $ q > p $.

\end{itemize}

\end{lem}

The following inequality is a very important tool. Let $ M $ denote the Hardy-Littlewood maximal operator. Then there is the following result.   

\begin{lem}\label{l_ineq1}

Let $ 1 < p \leq u < \infty    $. $ f $ is a locally Lebesgue-integrable function on $ \R $. Then there is a constant $ C > 0 $ independent of $ f  $ such that  
\begin{align*}
\Vert Mf \vert \mathcal{M}^{u}_{p}( \mathbb{R}^d) \Vert \leq C \Vert f \vert \mathcal{M}^{u}_{p}( \mathbb{R}^d) \Vert. 
\end{align*}
\end{lem} 

\begin{proof}
This result can be found in \cite{ChiFra}, see also theorem 6.19 in \cite{SawBook}. 
\end{proof}
 
Before we can write down the next result we have to introduce some additional notation. Let $ \nu \in \mathbb{R} $. Then $ H_{2}^{\nu}  (\R) $ denotes a Bessel-potential space with 
\begin{center}
$ \Vert f \vert H_{2}^{\nu}  (\R) \Vert =  \Vert (1 + \vert \xi \vert ^{2} )^{\frac{\nu}{2}} ( \cf f )(\xi )  \vert L_{2}(\R)  \Vert $ .
\end{center}

Now we are able to give the next lemma which can be found in \cite{SawTan}, see theorem 2.4. Moreover the following lemma is a special case of theorem 2.7 in \cite{TangXu}. Here a slightly different formulation is used. 

\begin{lem}\label{l_ineq2}

Let $ 0 < p \leq u < \infty $ and $ \nu > \frac{d}{\min(1,p)} + \frac{d}{2}   $. Let $ h \in H_{2}^{\nu}  (\R)$ and $ R > 0 $.  $ f \in \mathcal{M}^{u}_{p}( \mathbb{R}^d)   $ is a function with $ f \in \mathcal{S}'(\R) $ and $ \supp \cf f \subset B(0,R)$. Then there is a constant $ C > 0 $ independent of $ R, h $ and $ f $ such that
\begin{align*}
\Big \Vert  (2 \pi)^{\frac{d}{2}} \int_{\R} \cfi h(x - y) f(y) dy   \Big \vert \mathcal{M}^{u}_{p}( \mathbb{R}^d) \Big \Vert \leq C \Vert h(R \cdot ) \vert H_{2}^{\nu}  (\R) \Vert \cdot  \Vert f   \vert \mathcal{M}^{u}_{p}( \mathbb{R}^d)  \Vert.
\end{align*}
  
\end{lem}

It is also possible to define the Besov-Morrey spaces on domains. Here we want to work with smooth and bounded domains only. We use the following definition, whereby $ \mathcal{D}'(\Omega)  $ denotes the space of distributions on $  \Omega \subset \R $ as usual. 

\begin{defi}\label{def_bmD}

Let $ 0 < p \leq u < \infty $, $ 0 < q \leq \infty $ and $ s \in \mathbb{R} $. Let $ \Omega \subset \R$ be a bounded $ C^{\infty} $ domain. Then we define

\begin{center}
$ \mathcal{N}^{s}_{u,p,q}(\Omega) = \left \lbrace f \in \mathcal{D}'(\Omega) \; : \; f = g \; \mbox{in} \; \; \Omega \; \; \mbox{for some} \; g \in       \mathcal{N}^{s}_{u,p,q}(\R)  \right \rbrace $
\end{center}

and

\begin{center}
$ \Vert f \vert \mathcal{N}^{s}_{u,p,q}(\Omega) \Vert = \inf \left \lbrace    \Vert g \vert \mathcal{N}^{s}_{u,p,q}(\R) \Vert \; : \; f = g \; \mbox{in} \; \; \Omega \; \; \mbox{for} \; g \in \mathcal{N}^{s}_{u,p,q}(\R) \right \rbrace. $
\end{center}
 
\end{defi}

The properties of Besov-Morrey spaces on domains were investigated in \cite{Saw}. For our purposes the following result of Haroske and Skrzypczak that can be found in \cite{HaSkSM} is of special interest. 

\begin{lem}\label{l_bmD1}

Let $ 0 < p \leq u < \infty $, $ 0 < q \leq \infty   $ and $ 1 \leq v < \infty $ with $ p < v $. Let $ s \in \mathbb{R} $. Then the embedding $ \mathcal{N}^{s}_{u,p,q}(\Omega) \hookrightarrow L_{v}(\Omega) $ implies 
\begin{align*}
s \geq   d  \frac{p}{u} \left ( \frac{1}{p} - \frac{1}{v}    \right ).  
\end{align*} 
\end{lem} 

\begin{proof}
This result can be found in \cite{HaSkSM}, see proposition 5.3. The special case $ p = u $ can be found in \cite{RS}, see also \cite{sitr}.  
\end{proof}

Also the following tool is very useful. Let $ C(\R) $ be the space of all complex-valued bounded uniformly continuous functions on $ \R $. For $  m \in \mathbb{N} $ we put 

\begin{center}
$ C^{m}(\R) = \lbrace f : D^{\alpha}f \in C(\R) \; \mbox{for all} \; \vert \alpha \vert \leq m \rbrace  $.
\end{center}

\begin{lem}\label{l_bmD2}

Let $ 0 < p \leq u < \infty $, $ 0 < q \leq \infty $ and $ s \in \mathbb{R}  $. Let $ m \in \mathbb{N} $ be sufficiently large. Then there exists a positive constant $ C(m) $ such that for all $ g \in C^{m}(\R) $ and for all $ f \in \mathcal{N}^{s}_{u,p,q}(\R) $ we have

\begin{center}
$\Vert f \cdot g \vert  \mathcal{N}^{s}_{u,p,q}(\R)   \Vert \leq C(m) \Big ( \sum_{\vert \alpha \vert \leq m}  \Vert D^{\alpha} g \vert L_{\infty}(\R) \Vert  \Big ) \; \Vert f  \vert  \mathcal{N}^{s}_{u,p,q}(\R)   \Vert.$
\end{center}

\end{lem}

\begin{proof}
This result can be found in \cite{HaSkSM}, see theorem 2.6. More details can be found in \cite{Saw}. A related result can be found in \cite{ysy}, see theorem 6.1.      
\end{proof}

It is possible to describe the Besov-Morrey spaces by means of atomic decompositions. Details concerning this topic can be found in \cite{Ros} and in \cite{SawTan}. A short summary of the main ideas also can be found in \cite{HaMoSk}. For $ j \in \mathbb{Z} $ and $ k \in \mathbb{Z}^{d}  $ we define the dyadic cube $ Q_{j,k} = 2^{-j}([0,1)^{d} + k) $. By $ \chi_{j,k} $ we denote the characteristic function of the cube $ Q_{j,k} $. For $ 0 < u < \infty $ we put $ \chi_{j,k}^{(u)} = 2^{\frac{jd}{u}} \chi_{j,k} $. Now we are able to define what atoms are.

\begin{defi}\label{def_at}
Let $ s \in \mathbb{R} $, $ 0 < p \leq u < \infty $ and $ 0 < q \leq \infty $. Let $ K \in \mathbb{N}_{0} $ and $ L \in \mathbb{N}_{0} \cup \{ -1 \} $. Then for $ j \in  \mathbb{N}_{0}  $ and $ k \in  \mathbb{Z}^{d}   $ a collection of $ L_{\infty}(\R)  $-functions $ a_{j,k} $  is a family of $ ( K , L ) $-atoms if there are constants $ C_{1} > 1 $ and $ C_{2} > 0 $ such that the following properties are fulfilled. 

\begin{itemize}
\item[(i)]
We have $ \supp a_{j,k} \subset C_{1} Q_{j,k} $.

\item[(ii)]
For $  \vert \alpha \vert \leq K $ all derivatives $ D^{\alpha} a_{j,k} $ exist and we have $ \Vert D^{\alpha} a_{j,k} \vert L_{\infty}(\R) \Vert \leq C_{2} 2^{j \vert \alpha \vert}   $.

\item[(iii)]
For $ \vert \beta \vert \leq L $ we have $ \int_{\R} x^{\beta} a_{j,k}(x) dx = 0  $. In the case $ L = -1  $ this condition is empty.

\end{itemize}

\end{defi}

Moreover we are able to define the following sequence spaces.

\begin{defi}\label{def_seq_sp}
Let $ s \in \mathbb{R} $, $ 0 < p \leq u < \infty $ and $ 0 < q \leq \infty $. Then the sequence space $ {\bf n}^{s}_{u,p,q}(\R)  $ is defined to be the set of all sequences $ \lambda = \{ \lambda_{j,k}   \}_{j \in \mathbb{N}_{0} , k \in \mathbb{Z}^{d}} \subset \mathbb{C}  $ such that
\begin{align*} 
\Vert \lambda \vert {\bf n}^{s}_{u,p,q}(\mathbb{R}^{d})  \Vert :=  \Big ( \sum_{j = 0}^{\infty} 2^{jq ( s - \frac{d}{u}  )} \Big   \Vert \sum_{k \in \mathbb{Z}^{d}} \vert \lambda_{j,k} \vert \chi_{j,k}^{(u)}(x)  \Big \vert \mathcal{M}^{u}_{p}(\R) \Big  \Vert  ^{q} \Big   )^{\frac{1}{q}} < \infty .
\end{align*}

In the case $ q = \infty $ the usual modifications have to be made. 

\end{defi}

Using this notation we can formulate the following very important result that can be found in \cite{Ros}, see theorems 2.30 and 2.36, or in \cite{SawTan}, see corollary 4.10 and theorem 4.12. 

\begin{lem}\label{lem_atom}
Let $ s \in \mathbb{R} $, $ 0 < p \leq u < \infty $ and $ 0 < q \leq \infty $. Let $ K \in \mathbb{N}_{0} $ and $ L \in \mathbb{N}_{0} \cup \{ -1 \} $ such that $ K \geq \max(0 , s + 1 ) $ and $ L \geq \max( - 1 , \sigma_{p} - s )   $. Then there exists a constant $ C > 0 $ such that for all families $ \{ a_{j,k}   \}_{j \in \mathbb{N}_{0} , k \in \mathbb{Z}^{d}} $ of $  ( K , L ) $-atoms and all sequences $ \lambda = \{ \lambda_{j,k}   \}_{j \in \mathbb{N}_{0} , k \in \mathbb{Z}^{d}} \in {\bf n}^{s}_{u,p,q}(\R)   $ we have
\begin{align*}
\Big \Vert  \sum_{ j = 0}^{\infty} \sum_{k \in \mathbb{Z}^{d}} \lambda_{j,k}  a_{j,k}  \Big \vert \mathcal{N}^{s}_{u,p,q}(\R) \Big \Vert \leq C \Vert \lambda \vert {\bf n}^{s}_{u,p,q}(\mathbb{R}^{d})  \Vert.
\end{align*} 

\end{lem}

\section{A characterization of Besov-Morrey spaces in terms of differences}

In this chapter we want to prove characterizations of the Besov-Morrey spaces $ \mathcal{N}^{s}_{u,p,q}(\R)  $ that only use differences. For that purpose we want to use some ideas from Hedberg and Netrusov, see \cite{HN}. Because of this at first we have to introduce some additional notation that also can be found in chapter one of \cite{HN}.  

\begin{defi}\label{def_HN1}

Let $ E $ be a quasi - Banach space of sequences of Lebesgue-measurable functions on $ \mathbb{R}^d$. On $ E $ there is a non-negative function $ \Vert \cdot \Vert_{E} $, which has the same properties as a norm, except for the triangle inequality, and in addition satisfies the following conditions.

\begin{itemize}
\item[(i)] 
The metric space  $ ( E, \Vert \cdot \Vert_{E} ) $ is complete.

\item[(ii)]
If $ \{f_{i}\}_{i=0}^{\infty} \in E $ and $ \{g_{i} \}_{i=0}^{\infty} $ is a sequence of measurable functions such that $ | g_{i}  | \leq | f_{i} | $ a. e.  for all  $ i \in \mathbb{N}_{0}$, it follows that $ \{g_{i}\}_{i=0}^{\infty} \in E $ and  $ \left\| \left\{g_{i}\right\}_{i=0}^{\infty} \right\| _{E} \leq \left\| \left\{f_{i}\right\}_{i=0}^{\infty} \right\| _{E}$.

\item[(iii)]
There exist constants  $\kappa$ with $ 0 < \kappa \leq 1$ and $ C_{E} \geq 1 $, such that for any family $ \lbrace F_{i}\rbrace_{i=0}^{j}$  of elements in $ E $ one has the inequality
\begin{align*}
\Big \Vert \sum_{i=0}^{j}F_{i}   \Big \Vert_{E}^{\kappa} \leq C_{E} \sum_{i=0}^{j} \Big \Vert F_{i}   \Big \Vert_{E}^{\kappa}.
\end{align*}
 
\end{itemize} 

\end{defi}

Based on this definition Hedberg and Netrusov defined classes $ S( \epsilon_{+} , \epsilon_{-} , r ) $ of spaces $ E $ with $ \epsilon_{+} , \epsilon_{-} \in \mathbb{R}$ and $ 0 < r < \infty $. To do so for a sequence of functions $ \{f_{i}\}_{i=0}^{\infty} $ we define the left shift $ S_{+}$ and the right shift $ S_{-} $ by $ S_{+}\left(\left\{f_{i}\right\}_{i=0}^{\infty}\right) = \left\{f_{i+1}\right\}_{i=0}^{\infty} $ and $ S_{-}\left(\left\{f_{i}\right\}_{i=0}^{\infty}\right) = \left\{f_{i-1}\right\}_{i=0}^{\infty} $ with $ f_{-1}=0 $. Moreover for $ 0 < r < \infty $ and $ t \geq 0 $ we define the maximal function $ M_{r,t}f $ and the operator $ \hat{M}_{r,t}$ by
\begin{align*}
\hat{M}_{r,t}\left(\left\{f_{i}\right\}_{i=0}^{\infty}\right) = \left\{M_{r,t}f_{i}\right\}_{i=0}^{\infty}  = \Big \{ \sup_{a>0} \Big ( a^{-d} \int_{B(0,a)}  \frac{ | f_{i}(x+y)|^{r}}{(1+|y|)^{rt}}dy\Big)^{\frac{1}{r}}    \Big \}_{i=0}^{\infty}.
\end{align*}

\begin{defi}\label{def_HN2}

Let $ \epsilon_{+} , \epsilon_{-} \in \mathbb{R}$, $ 0 < r < \infty $ and $ t \geq 0 $. We define a class $ S( \epsilon_{+}, \epsilon_{-}, r, t)$ of spaces $ E $ by saying that $ E \in S( \epsilon_{+}, \epsilon_{-}, r, t ) $ if the following conditions are satisfied.

\begin{itemize}
\item[(i)] 
The linear operators $ S_{+} $ and $ S_{-}$  are continuous on $ E $ and there are constants $ C_{1}, C_{2} > 0 $ independent of $j$ and  $ \left\{f_{i}\right\}_{i=0}^{\infty} $ such that for all $ j \in \mathbb{N}$ we have

\begin{center}
$ \| ( S_{+} )^{j}  \vert \cl (E) \| \leq C_{1} 2^{-j \epsilon_{+}} \qquad $ and $ \qquad   \| ( S_{-} )^{j}  \vert \cl (E) \| \leq C_{2} 2^{j \epsilon_{-}}$.
\end{center}

\item[(ii)]

The operator $ \hat{M}_{r,t}$ is bounded on $ E $. There is a constant $ C > 0 $ independent of $ \left\{f_{i}\right\}_{i=0}^{\infty} $ such that $ \left\| \left\{M_{r,t}f_{i}\right\}_{i=0}^{\infty} \right\| _{E} \leq C \left\| \left\{f_{i}\right\}_{i=0}^{\infty} \right\| _{E}$ .

\end{itemize}

Set $  S( \epsilon_{+}, \epsilon_{-}, r ) = \bigcup_{t \geq 0}  S( \epsilon_{+}, \epsilon_{-},r ,t )$.

\end{defi}

Using this definitions we are able to define function spaces denoted by $ Y(E) $. 

\begin{defi}\label{def_HN3}

Let $ \epsilon_{+}, \epsilon_{-} \in \mathbb{R} $ and $ r > 0 $. Let $ E \in S( \epsilon_{+}, \epsilon_{-}, r  ) $. The space $ Y(E) $ consists of all distributions $f \in \mathcal{S}'(\R) $ which have a representation $ f = \sum_{i=0}^{\infty} f_{i} $ converging in $\mathcal{S}'(\R)$ such that we have $ \left\| \left\{f_{i}\right\}_{i=0}^{\infty} \right\| _{E} < \infty $, $ \supp \cf f_{0} \subset B(0,2) $ and $ \supp \cf f_{i} \subset B(0,2^{i+1})\setminus~B(0,2^{i-1}) $ for all $ i \in \mathbb{N}$.
\end{defi}

We put 
\begin{align*}
\Vert f \Vert_{Y(E)} := \inf \left\| \left\{f_{i}\right\}_{i=0}^{\infty} \right\| _{E}\, , 
\end{align*}
where the infimum is taken over all admissible representations of $ f $ as described in definition \ref{def_HN3}. Then $ \Vert f \Vert_{Y(E)} $ is a quasi-norm and $ Y(E) $ becomes a quasi-normed space. For the space $ Y(E) $ Hedberg and Netrusov proved a characterization in terms of differences, see proposition 1.1.12. and theorem 1.1.14. in \cite{HN}. This will be very important for the proof of the following result. 

\begin{prop}\label{pro_HNres1}

Let $ 0 < p \leq u < \infty $ and $ 0 < q \leq \infty $. Let $ 0 < v \leq \infty $ and $ s ~ > ~ d \max  ( 0, \frac{1}{p} - 1, \frac{1}{p} - \frac{1}{v}   ) $. Let $ N \in \mathbb{N} $ with $ N > s $. Then a function $ f~ \in ~ L_{p}^{loc}(\R)$ belongs to $ \mathcal{N}^{s}_{u,p,q}(\R)   $ if and only if $ f \in L_{v}^{loc}(\R)$ and ( modifications if $ q = \infty $ and/or $ v = \infty $) 
\begin{align*}
& \Vert f \vert \mathcal{N}^{s}_{u,p,q}(\R) \Vert^{(\clubsuit)} :=  \Big ( \Big \Vert  \Big ( \int_{B(x,1)}\vert f(y) \vert^{v} dy \Big )^{\frac{1}{v}}   \Big \vert  \mathcal{M}^{u}_{p}( \mathbb{R}^d)   \Big \Vert^{q} \\ 
& \qquad \qquad \qquad \qquad \qquad \quad + \sum_{j = 1}^{\infty} 2^{jq(s + \frac{d}{v})} \Big \Vert  \Big ( \int_{B(0, 2^{-j})}\vert \Delta^{N}_{h}f(x) \vert^{v} dh \Big )^{\frac{1}{v}} \Big \vert  \mathcal{M}^{u}_{p}( \mathbb{R}^d) \Big \Vert^{q} \Big )^{\frac{1}{q}}    
\end{align*}
is finite. The quasi-norms $ \Vert f \vert \mathcal{N}^{s}_{u,p,q}(\R) \Vert  $ and $ \Vert f \vert \mathcal{N}^{s}_{u,p,q}(\R) \Vert^{(\clubsuit)}    $  are equivalent for $ f \in L_{p}^{loc}(\R)$.

\end{prop}

\begin{rem}\label{rem_HNrs1}
In the formulation of proposition \ref{pro_HNres1} it is possible to replace the quasi-norm $ \Vert f \vert \mathcal{N}^{s}_{u,p,q}(\R) \Vert^{(\clubsuit)}   $ by
\begin{align*}
& \Vert f \vert \mathcal{N}^{s}_{u,p,q}(\R) \Vert^{(\spadesuit)} \\ 
& \qquad =  \Big (  \Vert  f    \vert  \mathcal{M}^{u}_{p}( \mathbb{R}^d)  \Vert^{q}  + \sum_{j = 1}^{\infty} 2^{jq(s + \frac{d}{v})} \Big \Vert  \Big ( \int_{B(0, 2^{-j})}\vert \Delta^{N}_{h}f(x) \vert^{v} dh \Big )^{\frac{1}{v}} \Big \vert  \mathcal{M}^{u}_{p}( \mathbb{R}^d) \Big \Vert^{q} \Big )^{\frac{1}{q}}. 
\end{align*}
To prove this we can use the ideas from remark 3.6. in \cite{Ho}.
\end{rem}

\begin{proof}
To prove this result we proceed like it is described in chapter 3 of \cite{Ho}. Here a characterization in terms of differences for the Triebel-Lizorkin-Morrey spaces $ \mathcal{E}^{s}_{u,p,q}(\R) $ is proved. Fortunately the spaces $  \mathcal{N}^{s}_{u,p,q}(\R) $ and $ \mathcal{E}^{s}_{u,p,q}(\R) $ have many properties in common. So we can use the ideas from chapter 3 of \cite{Ho} with some minor modifications to prove proposition \ref{pro_HNres1}. At first for $ s \in \mathbb{R} $, $  0 < p \leq u < \infty $ and $ 0 < q \leq \infty $ we define the following space of locally Lebesgue-integrable functions on $ \R $ ( modifications if $ q = \infty $ ). 
\begin{align*}
l_{q}^{s}(\mathcal{M}^{u}_{p}(\R)) = \Big \lbrace \lbrace f_{j} \rbrace_{j = 0 }^{\infty}   : \! \Vert \lbrace f_{j} \rbrace_{j = 0 }^{\infty} \Vert_{l_{q}^{s}(\mathcal{M}^{u}_{p}(\R))}  :=  \Big ( \sum_{j = 0}^{\infty} 2^{jsq} \Vert f_{j} \vert \mathcal{M}^{u}_{p}(\R)  \Vert^{q}  \Big )^{\frac{1}{q}}  < \infty  \Big \rbrace  
\end{align*}
It is not difficult to see that the pair $ ( l_{q}^{s}(\mathcal{M}^{u}_{p}(\R)) ,  \Vert \cdot \Vert_{l_{q}^{s}(\mathcal{M}^{u}_{p}(\R))} ) $ has all the properties that are written down in definition \ref{def_HN1}. Moreover for $ s \in \mathbb{R} $, $ 0 < r < p \leq u < \infty $, $ 0 < q \leq \infty $ and $ t \geq 0 $ we have $  l_{q}^{s}(\mathcal{M}^{u}_{p}(\R)) \in S(s,s,r,t) $, see definition \ref{def_HN2}. The main tool to prove this is lemma \ref{l_ineq1}. Next we investigate the space $ Y( l_{q}^{s}(\mathcal{M}^{u}_{p}(\R)) )  $, see definition \ref{def_HN3}. For $ s \in \mathbb{R} $, $ 0 < r < p \leq u < \infty  $ and $ 0 < q \leq \infty $ we obtain $ Y( l_{q}^{s}(\mathcal{M}^{u}_{p}(\R)) ) =  \mathcal{N}^{s}_{u,p,q}(\R) $. Moreover the quasi-norms $  \Vert \cdot \Vert_{Y(l_{q}^{s}(\mathcal{M}^{u}_{p}(\R)))} $ and $ \Vert \cdot \vert \mathcal{N}^{s}_{u,p,q}(\mathbb{R}^{d})  \Vert $ are equivalent. To prove this we follow the ideas of the proof from proposition 3.4 in \cite{Ho}. Here we have to use lemma \ref{l_ineq2}. Now we are able to get a characterization in terms of differences for the spaces $ \mathcal{N}^{s}_{u,p,q}(\mathbb{R}^{d}) $ like it is written down in the formulation of proposition \ref{pro_HNres1}. For that purpose we have to use proposition 1.1.12. and theorem 1.1.14. from \cite{HN}. In a first step we get the above characterization not for functions $ f~ \in ~ L_{p}^{loc}(\R) $ but for $ f \in L_{r}^{loc}(\R)$ with $ \max   ( \frac{d}{s + d}, \frac{d}{s + \frac{d}{v}}   ) < r < p $. We have $   L_{p}^{loc}(\R) \subset  L_{r}^{loc}(\R) $. So $ f \in  L_{p}^{loc}(\R) $ implies $ f \in L_{r}^{loc}(\R)  $ and the result can be obtained. So the proof is complete. Notice that a result similar to proposition \ref{pro_HNres1} also can be found in \cite{ysy}, see corollary 4.12.
\end{proof}

Now we are well prepared to prove the following main result.

\begin{satz}\label{thm_MR_va}
Let $ 0 < p \leq u < \infty $ and $ 0 < q \leq \infty $. Let $ 0 < v \leq \infty $ and $ 1 \leq a \leq \infty  $. Let 
\begin{align*}
s ~ > ~ d \max \Big  ( 0, \frac{1}{p} - 1, \frac{1}{p} - \frac{1}{v}  \Big ) . 
\end{align*}
Let $ N \in \mathbb{N} $ with $ N > s $. Then a function $ f~ \in ~ L_{p}^{loc}(\R)$ belongs to $ \mathcal{N}^{s}_{u,p,q}(\R)   $ if and only if $ f \in L_{v}^{loc}(\R)$ and (modifications if $ q = \infty $ and/or $ v = \infty $) 
\begin{align*}
& \underbrace{\Vert f \vert \mathcal{M}^{u}_{p}( \mathbb{R}^d)   \Vert + \Big ( \int_{0}^{a} t^{-sq-d \frac{q}{v}} \Big \Vert  \Big ( \int_{B(0,t)}\vert \Delta^{N}_{h}f(x) \vert^{v} dh \Big )^{\frac{1}{v}} \Big \vert  \mathcal{M}^{u}_{p}( \mathbb{R}^d) \Big \Vert^{q}  \frac{dt}{t} \Big )^{\frac{1}{q}}  } < \infty.  \\
& \hspace{4,5 cm} \Vert f \vert \mathcal{N}^{s}_{u,p,q}(\R) \Vert^{(v,a)} :=  
\end{align*}
The quasi-norms $ \Vert f \vert \mathcal{N}^{s}_{u,p,q}(\R) \Vert  $ and $ \Vert f \vert \mathcal{N}^{s}_{u,p,q}(\R) \Vert^{(v,a)} $  are equivalent for $ f \in L_{p}^{loc}(\R)$.

\end{satz}

\begin{rem}\label{rem_va}
Let $ 0 < v \leq \infty $ and $ 1 \leq a \leq \infty $. Then the letters $ v $ and $ a $ in the abbreviation $(v,a)$ indicate the dependence of the concrete quasi-norm $ \Vert \cdot \vert \mathcal{N}^{s}_{u,p,q}(\R) \Vert^{(v,a)} $ on these parameters. 
\end{rem}

\begin{proof}
To prove theorem \ref{thm_MR_va} we can use the ideas that can be found in the proofs of proposition 4.1 and theorem 6 from \cite{Ho}. Almost all the techniques that are described there also can be used here. Only some minor modifications have to be made. 

{\em Step 1.} At first we will deal with the case $ a = 1 $.

{\em Substep 1.1.} We prove that there is a constant $ C > 0 $ independent of $ f \in L_{p}^{loc}(\R) $ such that $ \Vert f \vert \mathcal{N}^{s}_{u,p,q}(\R) \Vert^{(\spadesuit)} \leq C \Vert f \vert \mathcal{N}^{s}_{u,p,q}(\R) \Vert^{(v,1)}  $. But this is just a consequence of the monotonicity of $  \int_{B(0,t)}\vert \Delta^{N}_{h}f(x) \vert^{v} dh   $ in $ t $. 

{\em Substep 1.2.} Next we will prove that for $ f \in \mathcal{N}^{s}_{u,p,q}(\R) $ there is a constant $ C > 0 $ independent of $ f $ such that $ \Vert f \vert \mathcal{N}^{s}_{u,p,q}(\R) \Vert^{(v,1)} \leq C  \Vert f \vert \mathcal{N}^{s}_{u,p,q}(\R) \Vert^{(\spadesuit)}  $. To prove this again at first we use the monotonicity of $  \int_{B(0,t)}\vert \Delta^{N}_{h}f(x) \vert^{v} dh   $ in $ t $. Moreover we use that we can write
\begin{equation}\label{diff_form}
\Delta^{N}_{h}f(x) = \sum_{k = 0}^{N} (-1)^{N-k} { N \choose k } f ( x + k h ).
\end{equation}
Then we obtain
\begin{align*}
& \Big ( \int_{0}^{1} t^{-sq-d \frac{q}{v}} \Big \Vert  \Big ( \int_{B(0,t)}\vert \Delta^{N}_{h}f(x) \vert^{v} dh \Big )^{\frac{1}{v}} \Big \vert  \mathcal{M}^{u}_{p}( \mathbb{R}^d) \Big \Vert^{q}  \frac{dt}{t} \Big )^{\frac{1}{q}} \\ 
& \qquad \leq C_{1} \Vert f \vert \mathcal{N}^{s}_{u,p,q}(\R) \Vert^{(\spadesuit)}  + C_{1} \Big \Vert  \Big ( \int_{B(0,1)}\vert \Delta^{N}_{h}f(x) \vert^{v} dh \Big )^{\frac{1}{v}} \Big \vert  \mathcal{M}^{u}_{p}( \mathbb{R}^d) \Big \Vert \\
& \qquad \leq C_{2} \Vert f \vert \mathcal{N}^{s}_{u,p,q}(\R) \Vert^{(\spadesuit)}  + C_{2} \Big \Vert  \Big ( \int_{B(x,N)}\vert f(z) \vert^{v} dz \Big )^{\frac{1}{v}} \Big \vert  \mathcal{M}^{u}_{p}( \mathbb{R}^d) \Big \Vert.
\end{align*}
Next we cover the ball $ B(x,N) $ with $ ( 2N + 1 )^{d} $ small balls with radius one like it is described in the proof of proposition 4.1 in \cite{Ho} and use the translation-invariance of the Morrey spaces. Then proposition \ref{pro_HNres1} leads to 
\begin{align*}
\Big \Vert  \Big ( \int_{B(x,N)}\vert f(z) \vert^{v} dz \Big )^{\frac{1}{v}} \Big \vert  \mathcal{M}^{u}_{p}( \mathbb{R}^d) \Big \Vert & \leq C_{3} \Big \Vert  \Big ( \int_{B(x,1)}\vert f(z) \vert^{v} dz \Big )^{\frac{1}{v}} \Big \vert  \mathcal{M}^{u}_{p}( \mathbb{R}^d) \Big \Vert  \\
& \leq C_{4} \Vert f \vert \mathcal{N}^{s}_{u,p,q}(\R) \Vert^{(\spadesuit)}.
\end{align*}

{\em Step 2.} Now we will deal with the case $ a = \infty $.

{\em Substep 2.1.} Here at first we prove that there is a constant $ C > 0 $ independent of $ f \in L_{p}^{loc}(\R) $ such that $ \Vert f \vert \mathcal{N}^{s}_{u,p,q}(\R) \Vert^{(v,1)} \leq C \Vert f \vert \mathcal{N}^{s}_{u,p,q}(\R) \Vert^{(v, \infty )}  $. But of course this is obvious. 

{\em Substep 2.2.} Next we will prove that for $ f \in \mathcal{N}^{s}_{u,p,q}(\R) $ there is a constant $ C > 0 $ independent of $ f $ such that $ \Vert f \vert \mathcal{N}^{s}_{u,p,q}(\R) \Vert^{(v, \infty )} \leq C  \Vert f \vert \mathcal{N}^{s}_{u,p,q}(\R) \Vert^{(v,1)}  $. To prove this at first because of the monotonicity of $  \int_{B(0,t)}\vert \Delta^{N}_{h}f(x) \vert^{v} dh   $ in $ t $ and formula \eqref{diff_form} we obtain
\begin{align*}
& \Big ( \int_{1}^{ \infty } t^{-sq-d \frac{q}{v}} \Big \Vert  \Big ( \int_{B(0,t)}\vert \Delta^{N}_{h}f(x) \vert^{v} dh \Big )^{\frac{1}{v}} \Big \vert  \mathcal{M}^{u}_{p}( \mathbb{R}^d) \Big \Vert^{q}  \frac{dt}{t} \Big )^{\frac{1}{q}} \\
& \qquad \leq C_{1} \Vert f \vert \mathcal{M}^{u}_{p}( \mathbb{R}^d)   \Vert + C_{1} \Big ( \sum_{j = 1}^{\infty} 2^{-jsq} 2^{-jd \frac{q}{v}} \Big \Vert  \Big ( \int_{B(x,N 2^{j})}\vert f(z) \vert^{v} dz \Big )^{\frac{1}{v}} \Big \vert  \mathcal{M}^{u}_{p}( \mathbb{R}^d) \Big \Vert^{q}   \Big )^{\frac{1}{q}}.
\end{align*} 
Now we want to cover the ball $ B(x , N  2^{j}) $ with $ ( 2 N \cdot 2^{j} + 1 )^{d} $ small balls with radius one. Let $ i \in \lbrace 1, 2, \ldots, ( 2 N \cdot 2^{j} + 1 )^{d} \rbrace $ and $ w_{i} $ appropriate displacement vectors such that
\begin{align*}
\bigcup_{i = 1}^{( 2 N \cdot 2^{j} + 1 )^{d}} B( x + w_{i} , 1) \supset  B(x , N  2^{j}).
\end{align*} 
We get
\begin{align*}
& \Big ( \sum_{j = 1}^{\infty} 2^{-jsq} 2^{-jd \frac{q}{v}} \Big \Vert  \Big ( \int_{B(x,N 2^{j})}\vert f(z) \vert^{v} dz \Big )^{\frac{1}{v}} \Big \vert  \mathcal{M}^{u}_{p}( \mathbb{R}^d) \Big \Vert^{q}   \Big )^{\frac{1}{q}} \\
& \qquad \leq \Big ( \sum_{j = 1}^{\infty} 2^{-jsq} 2^{-jd \frac{q}{v}} \Big \Vert  \Big (  \sum_{i = 1}^{( 2 N \cdot 2^{j} + 1 )^{d}} \int_{B(x + w_{i} , 1 )}\vert f(z) \vert^{v} dz \Big )^{\frac{1}{v}} \Big \vert  \mathcal{M}^{u}_{p}( \mathbb{R}^d) \Big \Vert^{q}   \Big )^{\frac{1}{q}}.
\end{align*}
We put $ \mu = \min(p,v) $ and remember that we have $ s > d \max  ( 0, \frac{1}{p} - \frac{1}{v}   )  $. Then the translation-invariance of the Morrey spaces in combination with proposition \ref{pro_HNres1} and step 1 of this proof leads to 
\begin{align*}
& \Big ( \sum_{j = 1}^{\infty} 2^{-jsq} 2^{-jd \frac{q}{v}} \Big \Vert  \Big (  \sum_{i = 1}^{( 2 N \cdot 2^{j} + 1 )^{d}} \int_{B(x + w_{i} , 1 )}\vert f(z) \vert^{v} dz \Big )^{\frac{1}{v}} \Big \vert  \mathcal{M}^{u}_{p}( \mathbb{R}^d) \Big \Vert^{q}   \Big )^{\frac{1}{q}} \\
& \qquad \leq C_{2} \Big ( \sum_{j = 1}^{\infty} 2^{-jsq} 2^{-jd \frac{q}{v}}  2^{jd \frac{q}{\mu}}  \Big \Vert  \Big (  \int_{B(x  , 1 )}\vert f(z) \vert^{v} dz \Big )^{\frac{1}{v}} \Big \vert  \mathcal{M}^{u}_{p}( \mathbb{R}^d) \Big \Vert^{q}   \Big )^{\frac{1}{q}} \\
& \qquad \leq C_{3} \Big \Vert  \Big (  \int_{B(x  , 1 )}\vert f(z) \vert^{v} dz \Big )^{\frac{1}{v}} \Big \vert  \mathcal{M}^{u}_{p}( \mathbb{R}^d) \Big \Vert \leq C_{4} \Vert f \vert \mathcal{N}^{s}_{u,p,q}(\R) \Vert^{(v,1)}.
\end{align*}

{\em Step 3.} At last we look at the case $ 1 < a < \infty $. But here the proof is just a simple consequence of the things we did before. 
\end{proof}

Now it is really easy to prove theorem \ref{MR1}.

{\bf \textit{Proof of theorem \ref{MR1}}. }
To prove theorem \ref{MR1} we just have to use theorem \ref{thm_MR_va} with $ a = \infty $. Then we get exactly what we want.

\vspace{0,2 cm}

It is also possible to describe the Besov-Morrey spaces by a generalisation of a modulus of smoothness. 

\begin{satz}\label{thm_MoS}
Let $ 0 < p \leq u < \infty  $, $ 0 < q \leq \infty   $ and $ s > d \max  ( 0 , \frac{1}{p}  - 1  )  $. Let $ N \in \mathbb{N}  $ with $ N > s  $. Then a function $ f \in L_{\max(1,p)}^{loc}(\R)  $ belongs to $  \mathcal{N}^{s}_{u,p,q}(\R) $ if and only if
\begin{align*}
\Vert f \vert \mathcal{N}^{s}_{u,p,q}(\R) \Vert^{(\omega)} :=    \Vert f \vert \mathcal{M}^{u}_{p}( \mathbb{R}^d)   \Vert + \Big ( \int_{0}^{ \infty } t^{-sq} \Big [ \sup_{|h| \leq t } \Vert  \Delta^{N}_{h}f  \vert  \mathcal{M}^{u}_{p}( \mathbb{R}^d)  \Vert  \Big ]^{q}  \frac{dt}{t} \Big )^{\frac{1}{q}}   < \infty.  
\end{align*}
The quasi-norms $ \Vert f \vert \mathcal{N}^{s}_{u,p,q}(\R) \Vert  $ and $ \Vert f \vert \mathcal{N}^{s}_{u,p,q}(\R) \Vert^{(\omega)} $  are equivalent for $ f \in L_{\max(1,p)}^{loc}(\R) $. In the case $ q = \infty  $ the usual modifications have to be made.
\end{satz}

\begin{proof}
{\em Step 1.} At first we prove that for $  f \in L_{\max(1,p)}^{loc}(\R)  $ there is a constant $ C > 0  $ independent of $ f $ such that $ \Vert f \vert \mathcal{N}^{s}_{u,p,q}(\R) \Vert  \leq C \Vert f \vert \mathcal{N}^{s}_{u,p,q}(\R) \Vert^{(\omega)}  $. To prove this because of $ f \in L_{\max(1,p)}^{loc}(\R) \subset L_{p}^{loc}(\R) $ and $ s > d \max  ( 0 , \frac{1}{p}  - 1  )   $ we can apply theorem \ref{thm_MR_va} with $ a = \infty $ and $ v = p   $. Then we obtain 
\begin{align*}
& \Vert f \vert \mathcal{N}^{s}_{u,p,q}(\R) \Vert \\
& \quad \leq C_{1} \Vert f \vert \mathcal{M}^{u}_{p}( \mathbb{R}^d)   \Vert + C_{1} \Big ( \int_{0}^{\infty} t^{-sq-d \frac{q}{p}} \Big \Vert  \Big ( \int_{B(0,t)}\vert \Delta^{N}_{h}f(x) \vert^{p} dh \Big )^{\frac{1}{p}} \Big \vert  \mathcal{M}^{u}_{p}( \mathbb{R}^d) \Big \Vert^{q}  \frac{dt}{t} \Big )^{\frac{1}{q}}.
\end{align*}
Next we use Fubini's theorem and get
\begin{align*}
& \Big ( \int_{0}^{\infty} t^{-sq-d \frac{q}{p}} \Big [ \sup_{y \in \R, r > 0} \vert B(y,r) \vert^{\frac{1}{u}-\frac{1}{p}} \Big ( \int_{B(y,r)}  \int_{B(0,t)}\vert \Delta^{N}_{h}f(x) \vert^{p} dh \:  dx      \Big )^{\frac{1}{p}} \Big ]^{q}  \frac{dt}{t} \Big )^{\frac{1}{q}} \\
& \quad = \Big ( \int_{0}^{\infty} t^{-sq-d \frac{q}{p}} \Big [ \sup_{y \in \R, r > 0} \vert B(y,r) \vert^{\frac{1}{u}-\frac{1}{p}} \Big ( \int_{B(0,t)}  \int_{B(y,r)}  \vert \Delta^{N}_{h}f(x) \vert^{p} dx \:  dh      \Big )^{\frac{1}{p}} \Big ]^{q}  \frac{dt}{t} \Big )^{\frac{1}{q}} \\
& \quad \leq C_{2} \Big ( \int_{0}^{\infty} t^{-sq} \Big [ \sup_{y \in \R, r > 0} \vert B(y,r) \vert^{\frac{1}{u}-\frac{1}{p}} \Big ( \sup_{|h| \leq t}  \int_{B(y,r)}  \vert \Delta^{N}_{h}f(x) \vert^{p} dx       \Big )^{\frac{1}{p}} \Big ]^{q}  \frac{dt}{t} \Big )^{\frac{1}{q}} \\
& \quad \leq C_{3} \Big ( \int_{0}^{ \infty } t^{-sq} \Big [ \sup_{|h| \leq t } \Vert  \Delta^{N}_{h}f  \vert  \mathcal{M}^{u}_{p}( \mathbb{R}^d)  \Vert  \Big ]^{q}  \frac{dt}{t} \Big )^{\frac{1}{q}}. 
\end{align*}
So step 1 of the proof is complete.

{\em Step 2.}
Now we prove that for $  f \in L_{\max(1,p)}^{loc}(\R) \cap \mathcal{N}^{s}_{u,p,q}(\R)  $ there is a constant $ C > 0  $ independent of $ f $ such that $ \Vert f \vert \mathcal{N}^{s}_{u,p,q}(\R) \Vert^{(\omega)}  \leq C \Vert f \vert \mathcal{N}^{s}_{u,p,q}(\R) \Vert  $. To prove this at first again we can apply theorem \ref{thm_MR_va} with $ v = p   $ and get
\begin{align*}
\Vert f \vert \mathcal{M}^{u}_{p}( \mathbb{R}^d)   \Vert \leq \Vert f \vert \mathcal{N}^{s}_{u,p,q}(\R) \Vert^{(p, \infty)} \leq C_{1} \Vert f \vert \mathcal{N}^{s}_{u,p,q}(\R) \Vert.
\end{align*}
To deal with the term
\begin{align*}
 \Big ( \int_{0}^{ \infty } t^{-sq} \Big [ \sup_{|h| \leq t } \Vert  \Delta^{N}_{h}f  \vert  \mathcal{M}^{u}_{p}( \mathbb{R}^d)  \Vert  \Big ]^{q}  \frac{dt}{t} \Big )^{\frac{1}{q}}  
\end{align*}
in the following we will use some ideas from Triebel, see chapter 2.5.11. in \cite{Tr83}. At first we transform the integral concerning $ t $ into a sum. We obtain
\begin{align*}
& \Big ( \int_{0}^{ \infty } t^{-sq} \Big [ \sup_{|h| \leq t } \Vert  \Delta^{N}_{h}f  \vert  \mathcal{M}^{u}_{p}( \mathbb{R}^d)  \Vert  \Big ]^{q}  \frac{dt}{t} \Big )^{\frac{1}{q}} \\
& \qquad \leq C_{1} \Big ( \sum_{k = - \infty}^{\infty} 2^{ksq} \Big [ \sup_{|h| \leq 2^{-k} } \Vert  \Delta^{N}_{h}f  \vert  \mathcal{M}^{u}_{p}( \mathbb{R}^d)  \Vert  \Big ]^{q}  \Big )^{\frac{1}{q}} . 
\end{align*}
Now let $ (\varphi_{j})_{j \in \N_0 }$ be a smooth dyadic decomposition of the unity. We put $  \varphi_{j} = 0 $ for $ j < 0 $. Then because of $ s > d \max  ( 0 , \frac{1}{p}  - 1  )  $ and $  f \in  \mathcal{N}^{s}_{u,p,q}(\R)  $ for every $   k \in \mathbb{Z} $ we have 
\begin{align*}
f = \sum_{m = - \infty}^{\infty} \mathcal{F}^{-1}[\varphi_{k + m} \mathcal{F}f] 
\end{align*}
with convergence not only in $  \mathcal{S}'(\mathbb{R}^{d}) $ but also in $  \mathcal{M}^{u}_{p}( \mathbb{R}^d) $. Let $ \tau = \min(1,p,q) $. We get
\begin{align*}
& \Big ( \sum_{k = - \infty}^{\infty} 2^{ksq} \Big [ \sup_{|h| \leq 2^{-k} } \Vert  \Delta^{N}_{h}f  \vert  \mathcal{M}^{u}_{p}( \mathbb{R}^d)  \Vert  \Big ]^{q}  \Big )^{\frac{\tau}{q}} \\
& \qquad = \Big ( \sum_{k = - \infty}^{\infty} 2^{ksq} \Big [ \sup_{|h| \leq 2^{-k} } \Big \Vert  \Delta^{N}_{h} \Big ( \sum_{m = - \infty}^{\infty} \mathcal{F}^{-1}[\varphi_{k + m} \mathcal{F}f] \Big ) \Big  \vert  \mathcal{M}^{u}_{p}( \mathbb{R}^d) \Big  \Vert  \Big ]^{q}  \Big )^{\frac{\tau}{q}} \\
& \qquad \leq \sum_{m = - \infty}^{\infty} \Big ( \sum_{k = - \infty}^{\infty} 2^{ksq} \Big [ \sup_{|h| \leq 2^{-k} } \Big \Vert  \Delta^{N}_{h} \mathcal{F}^{-1}[\varphi_{k + m} \mathcal{F}f]  \Big  \vert  \mathcal{M}^{u}_{p}( \mathbb{R}^d) \Big  \Vert  \Big ]^{q}  \Big )^{\frac{\tau}{q}}.
\end{align*}
Next we split up the outer sum. We obtain
\begin{align*}
&  \sum_{m = - \infty}^{\infty} \Big ( \sum_{k = - \infty}^{\infty} 2^{ksq} \Big [ \sup_{|h| \leq 2^{-k} } \Big \Vert  \Delta^{N}_{h} \mathcal{F}^{-1}[\varphi_{k + m} \mathcal{F}f]  \Big  \vert  \mathcal{M}^{u}_{p}( \mathbb{R}^d) \Big  \Vert  \Big ]^{q}  \Big )^{\frac{\tau}{q}} \\
& \qquad = \sum_{m = - \infty}^{-1} \Big ( \sum_{k = - \infty}^{\infty} 2^{ksq} \Big [ \sup_{|h| \leq 2^{-k} } \Big \Vert  \Delta^{N}_{h} \mathcal{F}^{-1}[\varphi_{k + m} \mathcal{F}f]  \Big  \vert  \mathcal{M}^{u}_{p}( \mathbb{R}^d) \Big  \Vert  \Big ]^{q}  \Big )^{\frac{\tau}{q}} \\
& \qquad \qquad + \sum_{m = 0}^{\infty} \Big ( \sum_{k = - \infty}^{\infty} 2^{ksq} \Big [ \sup_{|h| \leq 2^{-k} } \Big \Vert  \Delta^{N}_{h} \mathcal{F}^{-1}[\varphi_{k + m} \mathcal{F}f]  \Big  \vert  \mathcal{M}^{u}_{p}( \mathbb{R}^d) \Big  \Vert  \Big ]^{q}  \Big )^{\frac{\tau}{q}}.
\end{align*}
In what follows at first we will deal with the case $ m < 0 $. Here we have to start with some preliminary considerations that also can be found in chapter 2.5.11. in \cite{Tr83}. So for every $ |h| \leq 1  $ and $ x \in \R  $ there is a constant $ C_{2} > 0 $ independent of $ f $ and $ x $ such that
\begin{align*}
\vert (\Delta^{N}_{2^{-k}h} \mathcal{F}^{-1}[\varphi_{k + m} \mathcal{F}f])(x) \vert \leq C_{2} 2^{-kN} \sup_{|x-y|\leq N 2^{-k}} \sum_{| \alpha | = N} \vert ( D^{\alpha} \mathcal{F}^{-1}[\varphi_{k + m} \mathcal{F}f])(y) \vert.
\end{align*}
Moreover for $ j \in \mathbb{Z}   $ and $  a > 0 $ we will define the function
\begin{align*}
(\varphi_{j}^{\ast}f)(x) = \sup_{y \in \R} \frac{\vert ( \mathcal{F}^{-1}[\varphi_{j} \mathcal{F}f])(x - y) \vert}{1 + ( 2^{j + 2}| y |)^{a}}.
\end{align*}
Notice that for $ j < 0 $ because of $ \varphi_{j} = 0  $ we also have $ \varphi_{j}^{\ast}f = 0 $. Then for $ | \alpha | = N    $ and $  y \in \R $ there is a constant $ C_{3} > 0  $ independent of $ f $ and $ y $ such that
\begin{align*}
\vert ( D^{\alpha} \mathcal{F}^{-1}[\varphi_{k + m} \mathcal{F}f])(y) \vert \leq C_{3} 2^{(k + m )N} (\varphi_{k + m}^{\ast}f)(y).
\end{align*}
When we use this estimates because of the properties of the function $ \varphi_{j}^{\ast}f $ we obtain  
\begin{align*}
& \sum_{m = - \infty}^{-1} \Big ( \sum_{k = - \infty}^{\infty} 2^{ksq} \Big [ \sup_{|h| \leq 2^{-k} } \Big \Vert  \Delta^{N}_{h} \mathcal{F}^{-1}[\varphi_{k + m} \mathcal{F}f]  \Big  \vert  \mathcal{M}^{u}_{p}( \mathbb{R}^d) \Big  \Vert  \Big ]^{q}  \Big )^{\frac{\tau}{q}} \\
&  \qquad \leq C_{4} \sum_{m = - \infty}^{-1} \Big ( \sum_{k = - \infty}^{\infty} 2^{ksq} 2^{mNq}   \Big \Vert   \sup_{|x-y|\leq N 2^{-k}}   (\varphi_{k + m}^{\ast}f)(y)  \Big  \vert  \mathcal{M}^{u}_{p}( \mathbb{R}^d) \Big  \Vert^{q}  \Big )^{\frac{\tau}{q}} \\
&  \qquad \leq C_{5} \sum_{m = - \infty}^{-1} 2^{ (N - s)  m \tau}  \Big ( \sum_{k = - \infty}^{\infty} 2^{(k + m)sq}     \Vert    \varphi_{k + m}^{\ast}f    \vert  \mathcal{M}^{u}_{p}( \mathbb{R}^d)   \Vert^{q}  \Big )^{\frac{\tau}{q}}.
\end{align*}
Next we put $ k + m = j  $. Then since $ N > s  $ we get
\begin{align*}
& \sum_{m = - \infty}^{-1} \Big ( \sum_{k = - \infty}^{\infty} 2^{ksq} \Big [ \sup_{|h| \leq 2^{-k} } \Big \Vert  \Delta^{N}_{h} \mathcal{F}^{-1}[\varphi_{k + m} \mathcal{F}f]  \Big  \vert  \mathcal{M}^{u}_{p}( \mathbb{R}^d) \Big  \Vert  \Big ]^{q}  \Big )^{\frac{\tau}{q}} \\
&  \qquad \leq C_{5} \sum_{m = - \infty}^{-1} 2^{ (N - s)  m \tau}  \Big ( \sum_{j = - \infty}^{\infty} 2^{jsq}     \Vert    \varphi_{j}^{\ast}f    \vert  \mathcal{M}^{u}_{p}( \mathbb{R}^d)   \Vert^{q}  \Big )^{\frac{\tau}{q}}\\
&  \qquad \leq C_{6}   \Big ( \sum_{j = 0}^{\infty} 2^{jsq}     \Vert    \varphi_{j}^{\ast}f    \vert  \mathcal{M}^{u}_{p}( \mathbb{R}^d)   \Vert^{q}  \Big )^{\frac{\tau}{q}}.
\end{align*}
Now let $ a > \frac{d}{p}  $. Then a modification of lemma 1.1.7. from \cite{HN} in combination with lemma \ref{l_ineq1} from this paper (see also theorem 1.1. in \cite{yyMF}) yields 
\begin{align*}
& \sum_{m = - \infty}^{-1} \Big ( \sum_{k = - \infty}^{\infty} 2^{ksq} \Big [ \sup_{|h| \leq 2^{-k} } \Big \Vert  \Delta^{N}_{h} \mathcal{F}^{-1}[\varphi_{k + m} \mathcal{F}f]  \Big  \vert  \mathcal{M}^{u}_{p}( \mathbb{R}^d) \Big  \Vert  \Big ]^{q}  \Big )^{\frac{\tau}{q}} \\
&  \qquad \leq C_{7}   \Big ( \sum_{j = 0}^{\infty} 2^{jsq}     \Vert   \mathcal{F}^{-1}[\varphi_{j} \mathcal{F}f]      \vert  \mathcal{M}^{u}_{p}( \mathbb{R}^d)   \Vert^{q}  \Big )^{\frac{\tau}{q}} \\
&  \qquad = C_{7} \Vert f \vert \mathcal{N}^{s}_{u,p,q}(\R) \Vert^{\tau}  .
\end{align*}
Next we will deal with the case $ m \geq 0  $. Here at first we can use
\begin{align*}
( \Delta^{N}_{h} \mathcal{F}^{-1}[\varphi_{k + m} \mathcal{F}f]) (x) = \sum_{l = 0}^{N} (-1)^{N-l} { N \choose l }  \mathcal{F}^{-1}[\varphi_{k + m} \mathcal{F}f]  ( x + l h ).
\end{align*}
Recall the translation-invariance of the Morrey spaces. Then we obtain
\begin{align*}
& \sum_{m = 0}^{\infty} \Big ( \sum_{k = - \infty}^{\infty} 2^{ksq} \Big [ \sup_{|h| \leq 2^{-k} } \Big \Vert  \Delta^{N}_{h} \mathcal{F}^{-1}[\varphi_{k + m} \mathcal{F}f]  \Big  \vert  \mathcal{M}^{u}_{p}( \mathbb{R}^d) \Big  \Vert  \Big ]^{q}  \Big )^{\frac{\tau}{q}} \\
& \qquad \leq C_{8} \sum_{m = 0}^{\infty} \Big ( \sum_{k = - \infty}^{\infty} 2^{ksq} \Big [ \sup_{|h| \leq 2^{-k} } \sum_{l = 0}^{N} \Big \Vert    \mathcal{F}^{-1}[\varphi_{k + m} \mathcal{F}f]  ( x + l h )  \Big  \vert  \mathcal{M}^{u}_{p}( \mathbb{R}^d) \Big  \Vert  \Big ]^{q}  \Big )^{\frac{\tau}{q}} \\
& \qquad \leq C_{9} \sum_{m = 0}^{\infty} 2^{- ms \tau} \Big ( \sum_{k = - \infty}^{\infty} 2^{ksq} 2^{msq}   \Vert    \mathcal{F}^{-1}[\varphi_{k + m} \mathcal{F}f]      \vert  \mathcal{M}^{u}_{p}( \mathbb{R}^d)  \Vert^{q}  \Big )^{\frac{\tau}{q}} .
\end{align*}
Now we put $ k + m = j  $. Then since $ s > 0  $ we find
\begin{align*}
& \sum_{m = 0}^{\infty} \Big ( \sum_{k = - \infty}^{\infty} 2^{ksq} \Big [ \sup_{|h| \leq 2^{-k} } \Big \Vert  \Delta^{N}_{h} \mathcal{F}^{-1}[\varphi_{k + m} \mathcal{F}f]  \Big  \vert  \mathcal{M}^{u}_{p}( \mathbb{R}^d) \Big  \Vert  \Big ]^{q}  \Big )^{\frac{\tau}{q}} \\
& \qquad \leq C_{9} \sum_{m = 0}^{\infty} 2^{- ms \tau} \Big ( \sum_{j = 0}^{\infty} 2^{jsq}    \Vert    \mathcal{F}^{-1}[\varphi_{j} \mathcal{F}f]      \vert  \mathcal{M}^{u}_{p}( \mathbb{R}^d)  \Vert^{q}  \Big )^{\frac{\tau}{q}} \\
& \qquad \leq C_{10}  \Big ( \sum_{j = 0}^{\infty} 2^{jsq}    \Vert    \mathcal{F}^{-1}[\varphi_{j} \mathcal{F}f]      \vert  \mathcal{M}^{u}_{p}( \mathbb{R}^d)  \Vert^{q}  \Big )^{\frac{\tau}{q}} \\
&  \qquad = C_{10} \Vert f \vert \mathcal{N}^{s}_{u,p,q}(\R) \Vert^{\tau}  .
\end{align*}

So the proof is complete.
\end{proof}

Characterizations in terms of differences can be used to investigate whether some functions belong to a space $ \mathcal{N}^{s}_{u,p,q}(\R) $ or not. 

\begin{lem}\label{lem_f_a_d}
Let $s > 0 $, $ 1 \leq p \leq u < \infty $ and $ 0 < q \leq \infty  $. Let $ \alpha < 0 $, $ \delta \geq 0  $ and $ \vartheta > 0  $ with $  \vartheta $ very small. $ \rho \in C_{0}^{\infty}(\R)  $ is a smooth cut-off function with $ \rho(x) = 1  $ for $ x \in B(0 , \vartheta)  $ and $ \rho(x) = 0  $ for $ |x| > 2 \vartheta  $. We put
\begin{align*}
f_{\alpha, \delta}(x) = \rho(x) |x|^{\alpha} ( - \ln |x|)^{- \delta}.
\end{align*}

\begin{itemize}
\item[(i)] Let $ \delta = 0  $. Then we have $ f_{\alpha, 0} \in \mathcal{N}^{s}_{u,p,q}(\R)   $ if and only if we have either $s < \frac{d}{u} + \alpha   $ or $ s =  \frac{d}{u} + \alpha $ and $ q = \infty $. 

\item[(ii)] Let $ \delta > 0 $. Then we have $ f_{\alpha, \delta} \in \mathcal{N}^{s}_{u,p,q}(\R)   $ if and only if we have either $  s < \frac{d}{u} + \alpha   $ or $ s =  \frac{d}{u} + \alpha  $ with $  \delta q > 1  $.

\end{itemize}

\end{lem}

\begin{proof} To prove sufficiency we use a version of theorem \ref{thm_MR_va} with $ v = p $ and $ a $ small. At first we transform the quasi-norm $  \Vert f_{\alpha, \delta} \vert \mathcal{N}^{s}_{u,p,q}(\R) \Vert^{(p,a)} $ like it is described in step 1 of the proof from theorem \ref{thm_MoS}. After this we proceed like it is explained in the proof of lemma 1 from chapter 2.3.1. in \cite{RS}. Also to prove necessity we can use the techniques that are described in the proof of this lemma. For that purpose in the case $ 0 < s < 1 $ we have to use theorem \ref{thm_MR_va} with $ v = 1  $, $ N = 1 $ and $ a $ small. In the case $ s \geq 1 $ at first we have to apply theorem 3.3. from \cite{HaSkCE}. Then we can proceed like in the case $ 0 < s < 1 $ and obtain the desired result.
\end{proof}

\section{Besov-Morrey spaces and differences: necessary conditions}

As you can see in theorem \ref{thm_MR_va} some conditions concerning the parameter $ s $ do appear. In detail the restrictions
\begin{equation}\label{s_res1}
s ~ > ~ d \max \left ( 0, \frac{1}{p} - 1, \frac{1}{p} - \frac{1}{v} \right ) \qquad \mbox{and} \qquad N > s 
\end{equation} 
can be found. In this chapter our main goal is to investigate whether these conditions are not only sufficient but also necessary. For this purpose we will define the following function spaces.

\begin{defi}\label{def_Dif_sp}
Let $ s \in \mathbb{R}  $, $ 0 < p \leq u < \infty $, $ 0 < q \leq \infty $, $ 0 < v \leq \infty $, $ 1 \leq a \leq \infty $ and  $ N \in \mathbb{N} $. Then $ {\bf N}^{s, N, a}_{u, p, q, v}(\R) $ is the collection of all $ f \in L_{\max(p,v)}^{loc}(\R)$ such that 
$ \Vert f \vert \mathcal{N}^{s}_{u,p,q}(\R)  \Vert^{(v,a)} $ is finite. 
\end{defi}

In what follows we will investigate in which cases we have $ \mathcal{N}^{s}_{u,p,q}(\R) \not =  {\bf N}^{s, N, a}_{u, p, q, v}(\R)  $. To answer this question we will use a lot of different techniques. Therefore it seems to be reasonable to examine each condition separately. In a first step we will deal with the condition $ s > 0 $.

\subsection{The necessity of  $ s > 0 $}

It will turn out that it is not possible to describe the Besov-Morrey spaces in terms of differences in the case $ s < 0 $. Also for $ s = 0 $ in many cases we can prove that the spaces $  \mathcal{N}^{s}_{u,p,q}(\R)  $ do not coincide with $ {\bf N}^{s, N, a}_{u, p, q, v}(\R)  $.

\begin{prop}\label{nec_0MR1}
Let $ s \in \mathbb{R}  $, $ 0 < p \leq u < \infty $, $ 0 < q \leq \infty $, $ 0 < v \leq \infty $ and $ N \in \mathbb{N} $ with $ N > s $. Then the following assertions are true.

\begin{itemize}
\item[(i)]
Let $ s < 0 $ and $ 1 \leq a \leq \infty $. Then we have $  \mathcal{N}^{s}_{u,p,q}(\R) \not =  {\bf N}^{s, N, a}_{u, p, q, v}(\R)   $. 

\item[(ii)]
Let $ s = 0 $, $ 1 \leq a \leq \infty $ and $ p \geq 2  $ with $ q > 2  $. Then we have $  \mathcal{N}^{0}_{u,p,q}(\R) \not =  {\bf N}^{0, N, a}_{u, p, q, v}(\R)    $.

\item[(iii)]
Let $ s = 0 $, $ 1 \leq a \leq \infty $ and $ 1 \leq p < 2  $ with $ q > p $. Then we have $  \mathcal{N}^{0}_{u,p,q}(\R) \not =  {\bf N}^{0, N, a}_{u, p, q, v}(\R)    $.

\item[(iv)]
Let $ s = 0 $ and $ a = \infty $. Then we have $  \mathcal{N}^{0}_{u,p,q}(\R) \not =  {\bf N}^{0, N, \infty}_{u, p, q, v}(\R)    $.

\end{itemize}

\end{prop} 

\begin{proof}
 
{\em Step 1.} At first we prove $ (i), (ii) $ and $ (iii)$. In each of this cases the spaces $ \mathcal{N}^{s}_{u,p,q}(\R) $ contain singular distributions, see lemma \ref{l_bp2}. So a characterization of $ \mathcal{N}^{s}_{u,p,q}(\R) $ in terms of differences is not possible.

{\em Step 2.} Now we look at the case $ s = 0 $ and $ a = \infty $. In the case $ q = \infty $ the spaces $ \mathcal{N}^{0}_{u,p,\infty}(\R) $ contain singular distributions, see lemma \ref{l_bp2}. So in what follows we can assume $ 0 < q < \infty $. We will use an idea from \cite{Ho}, see proposition 5.1. Let $ f \in C_{0}^{\infty}(\R) $ with $ f(x) = 1 $ for $ |x| \leq 1 $ and $ f(x) = 0 $ for $ |x| > 2 $. Then because of lemma \ref{l_bp1} we have $ f \in \mathcal{N}^{0}_{u,p,q}(\R)  $. But we are able to show that we have $ f \not \in   {\bf N}^{0, N, \infty}_{u, p, q, v}(\R)    $. At first instead of the supremum we choose the ball $ B(0,1) $. Then we get
\begin{align*}
\Vert f \vert \mathcal{N}^{0}_{u,p,q}(\R) \Vert^{(v, \infty)} & \geq C_{1} \Big ( \int_{0}^{ \infty} t^{-d \frac{q}{v}} \Big ( \int_{B(0,1)}  \Big ( \int_{B(0,t)}\vert \Delta^{N}_{h}f(x) \vert^{v} dh \Big )^{\frac{p}{v}} dx \Big )^{\frac{q}{p}}  \frac{dt}{t} \Big )^{\frac{1}{q}} \\
& \geq C_{1} \Big ( \int_{5}^{ \infty} t^{-d \frac{q}{v}} \Big ( \int_{B(0,1)}  \Big ( \int_{B(0,t) \setminus B(0,4) }\vert \Delta^{N}_{h}f(x) \vert^{v} dh \Big )^{\frac{p}{v}} dx \Big )^{\frac{q}{p}}  \frac{dt}{t} \Big )^{\frac{1}{q}}.
\end{align*}
Now we use formula \eqref{diff_form} that can be found in the proof of theorem \ref{thm_MR_va}. For $ x \in B(0,1) $ we have $ f(x) = 1 $. For $ x \in B(0,1) $, $ |h| \geq 4 $, $ k \geq 1 $ we get $ | kh + x | \geq | |kh| - |x|| \geq 3 $ and so $ f ( x + k h ) = 0   $. Hence we obtain $ \vert \Delta^{N}_{h}f(x) \vert = 1 $. Consequently we get
\begin{align*}
\Vert f \vert \mathcal{N}^{0}_{u,p,q}(\R) \Vert^{(v, \infty)} & \geq C_{1} \Big ( \int_{5}^{ \infty} t^{-d \frac{q}{v}} \Big ( \int_{B(0,1)}  \Big ( \int_{B(0,t) \setminus B(0,4) } 1 dh \Big )^{\frac{p}{v}} dx \Big )^{\frac{q}{p}}  \frac{dt}{t} \Big )^{\frac{1}{q}} \\
& \geq C_{2} \Big ( \int_{5}^{ \infty}   \frac{dt}{t} \Big )^{\frac{1}{q}} = \infty .
\end{align*}
Because of this we obtain $  \mathcal{N}^{0}_{u,p,q}(\R) \not =  {\bf N}^{0, N, \infty}_{u, p, q, v}(\R)  $ and the proof is complete.
\end{proof}

\subsection{Is the condition $ s > d  ( \frac{1}{p} - 1    )  $ necessary ? }

In this subsection we want to investigate whether the condition $ s > d  ( \frac{1}{p} - 1    )  $ that you can find in theorem \ref{thm_MR_va} is necessary. To give a satisfactory answer to this question seems to be not so easy. But some results are already known.

\begin{prop}\label{nec_p1MR1}
Let $ 0 < p \leq u < \infty $, $ 0 < q \leq \infty $, $ s \geq 0  $, $ 0 < v \leq \infty $ and $ 1 \leq a \leq \infty  $. We have $ N \in \mathbb{N} $ with $ N > s $. Let $ 0 < p < 1 $. Then the following assertions are true.

\begin{itemize}
\item[(i)]
Let $ s < d \frac{p}{u}  ( \frac{1}{p} - 1    ) $. Then we have $  \mathcal{N}^{s}_{u,p,q}(\R) \not =  {\bf N}^{s, N, a}_{u, p, q, v}(\R)   $. 

\item[(ii)]
Let $ s = d \frac{p}{u}  ( \frac{1}{p} - 1    ) $ and $ q > 1 $. Then we have $  \mathcal{N}^{s}_{u,p,q}(\R) \not =  {\bf N}^{s, N, a}_{u, p, q, v}(\R)    $.

\end{itemize}

\end{prop} 

\begin{proof}
In both cases the spaces $ \mathcal{N}^{s}_{u,p,q}(\R) $ contain singular distributions, see lemma \ref{l_bp2}. So a characterization of $ \mathcal{N}^{s}_{u,p,q}(\R) $ in terms of differences is not possible.
\end{proof}

\begin{prop}\label{nec_p_emb}
Let $  s > 0  $, $ 0 < p < u < \infty $, $ 0 < q \leq \infty  $ and $ N \in \mathbb{N} $ with $ N > s $. Let $ 0 < p < 1  $ and 
\begin{align*}
d \frac{p}{u} \left ( \frac{1}{p} - 1   \right ) < s \leq d  \left ( \frac{1}{p} - 1   \right ).
\end{align*}
Then for all $ f \in \mathcal{N}^{s}_{u,p,q}(\R) $ there is a $ C > 0  $ independent of $ f $ such that
\begin{align*}
\Vert f \vert \mathcal{N}^{s}_{u,p,q}(\R)  \Vert  \leq C   \Vert f \vert \mathcal{N}^{s}_{u,p,q}(\R)  \Vert^{(1, \infty )}.
\end{align*}
\end{prop}

\begin{proof}
To prove this result we can use the techniques that are described in step 2 of the proof from theorem 2.5.10. in \cite{Tr83}. The main tool for the proof is a classical construction from approximation theory that can be found in \cite{Ni}, see chapter 5.2.1. Here we apply the version from \cite{Si}, see the proof of lemma 10. Notice that it is important that we have $ s > d \frac{p}{u}  ( \frac{1}{p} - 1   )    $ which guarantees $ f \in L_{1}^{loc}(\R) \cap \mathcal{S}'(\mathbb{R}^{d}) $, see theorem 3.3. in \cite{HaMoSk}.
\end{proof}

\subsection{Is the condition $ s > d  ( \frac{1}{p} - \frac{1}{v}    )  $ necessary ? }

In this subsection we want to investigate whether the condition $ s > d  ( \frac{1}{p} - \frac{1}{v}   )  $ that you can find in theorem \ref{thm_MR_va} is necessary. The following result is not optimal but tells us that the condition $ s > d  ( \frac{1}{p} - \frac{1}{v}    )  $ is at least partly necessary. 

\begin{prop}\label{nec_pvMR1}
Let $ s \geq 0 $, $ 0 < p \leq u < \infty $ and $ 0 < q \leq \infty $. Let $ 0 < p < v < \infty $, $ 1 \leq a \leq \infty  $ and $ N \in \mathbb{N} $ with $ N > s $. We have $ s < d \frac{p}{u}  ( \frac{1}{p} - \frac{1}{v}   )$. Then we have $  \mathcal{N}^{s}_{u,p,q}(\R) \not =  {\bf N}^{s, N, a}_{u, p, q, v}(\R)   $. 
\end{prop}

\begin{proof}
{\em Step 1.} At first we look at the case $ 0 < v \leq 1$. Here $ s < d \frac{p}{u}  ( \frac{1}{p} - \frac{1}{v}    ) $ implies $ s < d \frac{p}{u}  ( \frac{1}{p} - 1    ) $. But from proposition \ref{nec_p1MR1} we know that in this case it is not possible to describe the spaces $ \mathcal{N}^{s}_{u,p,q}(\R)   $ in terms of differences. 

{\em Step 2.} Now we work with the case $ 1 < v < \infty $. We will use some ideas from \cite{HoSi}, see proposition 5.2.  Let $ s < d \frac{p}{u}  ( \frac{1}{p} - \frac{1}{v}    )  $. We will argue by contradiction. Our first assumption is 
$  \mathcal{N}^{s}_{u,p,q}(\R)  =  {\bf N}^{s, N, a}_{u, p, q, v}(\R) $ as sets. Then $ \mathcal{N}^{s}_{u,p,q}(\R)  $ can not contain singular distributions and $ \mathcal{N}^{s}_{u,p,q}(\R)  \subset L_1^{loc} (\R)$ follows. Our second assumption is a sharpening of the first one. We assume that the identity  
$Id: {\bf N}^{s, N, a}_{u, p, q, v}( B( 0 , \frac{1}{8N} ) )  \to \mathcal{N}^{s}_{u,p,q}( B( 0 , \frac{1}{8N} ) )$ is a continuous operator. Here $ {\bf N}^{s, N, a}_{u, p, q, v}( B( 0 , \frac{1}{8N} ) )$ is defined to be the set of all $f\in {\bf N}^{s, N, a}_{u, p, q, v}(\R)$ satisfying $\supp f \subset B( 0 , \frac{1}{8N})$.
We will first disprove assumption two, afterwards assumption one.   

{\em Substep 2.1.}  
We use $ f \in \mathcal{N}^{s}_{u,p,q}(\R) $ with $ \supp f \subset B(0, \frac{1}{4N}) $. We will prove that there is a $ C > 0 $ independent of $ f $ such that
\begin{equation}\label{prof_pv1}
\Vert f \vert L_{v}(\R) \Vert \leq C \Vert f \vert \mathcal{N}^{s}_{u,p,q}(\R) \Vert .
\end{equation}
Because of our assumption we can start with
\begin{align*}
\Vert f \vert \mathcal{N}^{s}_{u,p,q}(\R) \Vert & \geq C_{1} \Big ( \int_{0}^{a} t^{-sq-d \frac{q}{v}} \Big \Vert  \Big ( \int_{B(0,t)}\vert \Delta^{N}_{h}f(x) \vert^{v} dh \Big )^{\frac{1}{v}} \Big \vert  \mathcal{M}^{u}_{p}( \mathbb{R}^d) \Big \Vert^{q}  \frac{dt}{t} \Big )^{\frac{1}{q}}.
\end{align*}
Next instead of the supremum we choose the ball $ B(0,\frac{N + 1}{4})  $. Then we obtain
\begin{align*}
\Vert f \vert \mathcal{N}^{s}_{u,p,q}(\R) \Vert & \geq C_{2} \Big ( \int_{0}^{1} t^{-sq-d \frac{q}{v}} \Big ( \int_{B(0,\frac{N + 1}{4})}  \Big ( \int_{B(0,t)}\vert \Delta^{N}_{h}f(x) \vert^{v} dh \Big )^{\frac{p}{v}} dx \Big )^{\frac{q}{p}}  \frac{dt}{t} \Big )^{\frac{1}{q}}.
\end{align*}
Now we are exactly in the same situation as it is described in step 1 of the proof from proposition 5.2 in \cite{HoSi}. So we can proceed like there and get the desired result.

{\em Substep 2.2.} In this substep we will work with function spaces on smooth and bounded domains. They have been introduced in definition \ref{def_bmD}. As domain we choose the ball $ B( 0 , \frac{1}{8N} )  $. We want to prove that we have the continuous embedding 
\begin{equation}\label{prof_pv2}
\mathcal{N}^{s}_{u,p,q} \Big ( B( 0 , \frac{1}{8N} ) \Big ) \hookrightarrow L_{v} \Big ( B( 0 , \frac{1}{8N} ) \Big ) .
\end{equation}
To prove this we use the methods that are described in step 2 of the proof from proposition 5.2 in \cite{HoSi}. Only some minor modifications have to be made. Our main tools for the proof are formula \eqref{prof_pv1} and lemma \ref{l_bmD2}. With them it is easy to get the desired result. We omit the details. 

{\em Substep 2.3.} Next we want to use lemma \ref{l_bmD1}. From the substep before we know $ \mathcal{N}^{s}_{u,p,q}  ( B( 0 , \frac{1}{8N} )  ) \hookrightarrow L_{v}  ( B( 0 , \frac{1}{8N} ) ) $. Then because of $  p < v $ lemma \ref{l_bmD1} tells us that we have
\begin{align*}
s \geq   d  \frac{p}{u} \left ( \frac{1}{p} - \frac{1}{v}    \right ).
\end{align*}
But at the beginning of step 2 we said that we have $ s < d \frac{p}{u}  ( \frac{1}{p} - \frac{1}{v}   )  $. This is a contradiction. So our assumption on the continuity of the identity must have been wrong.

Let $(f_j)_j$ be a convergent sequence in ${\bf N}^{s, N, a}_{u, p, q, v}( B( 0 , \frac{1}{8N} ) )$
with limit $ f \in {\bf N}^{s, N, a}_{u, p, q, v}( B( 0 , \frac{1}{8N}))$. In addition we assume $\lim_{j \to \infty} f_j = g$ in 
$ \mathcal{N}^{s}_{u,p,q} ( B( 0 , \frac{1}{8N} ) )$. The first fact implies  convergence in $L_p (\R)$, see theorem \ref{thm_MR_va} and definition \ref{def_Dif_sp}. This yields convergence almost everywhere for an appropriate subsequence $(f_{j_\ell})_{\ell}$. 
The second assumption applied to this subsequence yields the existence of extensions $h_{j_\ell}$ of $f_{j_\ell} - g$ such that
\[
\lim_{\ell \to \infty} \|\,  h_{j_\ell}\,| \mathcal{N}^{s}_{u,p,q}(\R)\|=0\, .
\]
Without loss of generality we may assume $\supp h_{j_\ell} \subset [-1,1]^d$. Recall that we have assumed $ \mathcal{N}^{s}_{u,p,q}(\R)  =  {\bf N}^{s, N, a}_{u, p, q, v}(\R)  $. Then $ \mathcal{N}^{s}_{u,p,q}(\R)  $ can not contain singular distributions and $ \mathcal{N}^{s}_{u,p,q}(\R)  \subset L_1^{loc} (\R)$ follows. But then we also find $ \mathcal{N}^{s}_{u,p,q} ( B( 0 , \frac{1}{8N} ) ) \hookrightarrow L_{1}( B( 0 , \frac{1}{8N} ) )  $. Now because of $ \lim_{j \to \infty} f_j = g$ in 
$ \mathcal{N}^{s}_{u,p,q} ( B( 0 , \frac{1}{8N} ) )$ we obtain
\[
\lim_{\ell \to \infty}\,  \| \, f_{j_\ell} - g \, |L_1 (B(0,\frac{1}{8N}))\| =0.
\]
By switching to a further subsequence if necessary we conclude $f=g$ almost everywhere. So we have proved that the identity 
$Id: {\bf N}^{s, N, a}_{u, p, q, v}( B( 0 , \frac{1}{8N} ) )  \to \mathcal{N}^{s}_{u,p,q}( B( 0 , \frac{1}{8N} ) )$ is a closed linear operator. The Closed Graph Theorem, which remains to hold for quasi-Banach spaces, yields 
that $Id$ must be continuous. But this contradicts our previous conclusion.
Therefore also our assumption concerning the equality of the sets must be wrong.
This proves 
$  \mathcal{N}^{s}_{u,p,q}(\R) \not =  {\bf N}^{s, N, a}_{u, p, q, v}(\R) $ as claimed.
\end{proof} 

\begin{rem}\label{rem_v_ein}
Let $ 0 < p \leq u < \infty $, $ s > 0 $, $ 0 < q \leq \infty  $, $ 1 \leq a \leq \infty  $ and $ N \in \mathbb{N} $ with $ N > s  $.  Let $ 0 < v_{1} < v_{2} < \infty  $. Then we have
\begin{align*}
{\bf N}^{s, N, a}_{u, p, q, \infty}(\R) \hookrightarrow {\bf N}^{s, N, a}_{u, p, q, v_{2}}(\R) \hookrightarrow {\bf N}^{s, N, a}_{u, p, q, v_{1}}(\R).
\end{align*}
So for $ v > \max(1 , p )   $ and $  d \max  ( 0 , \frac{1}{p} - 1     ) < s \leq d  (  \frac{1}{p} - \frac{1}{v}     ) $ because of theorem \ref{thm_MR_va} we have
\begin{align*}
{\bf N}^{s, N, a}_{u, p, q, v}(\R) \hookrightarrow {\bf N}^{s, N, a}_{u, p, q, 1}(\R) = \mathcal{N}^{s}_{u,p,q}(\R).
\end{align*}
\end{rem}

\subsection{The necessity of $ N > s $ }

In this subsection we will prove that it is not possible to describe the Besov-Morrey spaces $ \mathcal{N}^{s}_{u,p,q}(\R) $ with differences of order $ N \in \mathbb{N} $ if we have $ N \leq s $. The following proposition is the main result of this subchapter. 

\begin{prop}\label{nec_N>sMR1}
Let $ 0 < p \leq u < \infty $, $ s \geq 0 $ and $ 0 < q \leq \infty  $. Let $ 0 < v \leq \infty $, $ 1 \leq a \leq \infty  $ and $ N \in \mathbb{N} $ with $ N \leq s  $. Then we have $  \mathcal{N}^{s}_{u,p,q}(\R) \not =  {\bf N}^{s, N, a}_{u, p, q, v}(\R)   $ if we are in one of the following cases.

\begin{itemize}
\item[(i)]
We have $ N < s $ and $ 0 < q \leq \infty $.

\item[(ii)]
We have $ N = s $ and $ 0 < q < \infty $.

\item[(iii)]
We have $ N = s $ with $ q = \infty  $ and $ u = p  $. Moreover we have $ v \geq 1  $.

\end{itemize}

\end{prop}  

\begin{proof}

{\em Step 1.}
At first we have either $ N < s $ and $ 0 < q \leq \infty $ or $ N = s $ and $ 0 < q < \infty $. We will use some ideas from the proof of proposition 5.5 in \cite{HoSi}. We work with a function $ f \in C_{0}^{\infty}(\R) $ that has a support in $ B(0, 3N + 3) $. In $ B(0, 2N + 2 ) $ this function looks like
\begin{equation}\label{prof_N>s1}
f(x_{1}, x_{2}, \ldots , x_{d}) = e^{x_{1} + x_{2} + x_{3} + \ldots + x_{d}}.
\end{equation} 
Then because of lemma \ref{l_bp1} we have $ f \in \mathcal{N}^{s}_{u,p,q}(\R)  $. We want to prove that we have $ \Vert f \vert \mathcal{N}^{s}_{u,p,q}(\R) \Vert^{( v , a )} = \infty  $. Let $ 0 < \varepsilon < 1 $. We define 
\begin{align*}
H^{d}_{+} = \lbrace h = ( h_{1} , h_{2} , \ldots , h_{d} ) \in \R \; : \; h_{1} \geq 0 ,  h_{2} \geq 0 , \ldots , h_{d} \geq 0 \rbrace.
\end{align*}
Then we obtain
\begin{align*}
&\Vert f \vert \mathcal{N}^{s}_{u,p,q}(\R) \Vert^{( v , a )} \\
& \qquad \geq C_{1}  \Big ( \int_{ \varepsilon }^{1} t^{-sq-d \frac{q}{v}} \Big ( \int_{B(0,1)}  \Big ( \int_{( B(0,t)\setminus B(0, \frac{t}{2}) ) \cap H^{d}_{+}}\vert \Delta^{N}_{h}f(x) \vert^{v} dh \Big )^{\frac{p}{v}} dx  \Big )^{\frac{q}{p}}  \frac{dt}{t} \Big )^{\frac{1}{q}} .
\end{align*}
Now we are exactly in the same situation as it is described in the proof of proposition 5.5 in \cite{HoSi}. So we can proceed like there. At the end we obtain that $ \Vert f \vert \mathcal{N}^{s}_{u,p,q}(\R) \Vert^{( v , a )} $ tends to infinity if $ \varepsilon $ tends to zero. So step one of this proof is complete. 

{\em Step 2.} Now we will deal with the case $ s = N $ and $ q = \infty $ with $ p = u  $ and $ v \geq 1  $. Here we will use some ideas from Peter Oswald, see \cite{Os}. We fix $ r \in \mathbb{N} $ with $ r > 4 $ such that $ 2^{r + 1} \geq N + 4  $. Let $ \phi \in C_{0}^{\infty}(\R)  $ be a function with $ \supp \phi \subset B(0,1) \cap [0,1)^{d} $ such that $  \phi ( \cdot - 32^{r - 2} ( 1, 1, \ldots, 1)^{T}  )   $ fulfills moment conditions up to order $ L \in \mathbb{N}_{0} \cup \left\{ -1 \right\} $ with $ L \geq \max( - 1 , \sigma_{p} - N )   $. Moreover there is a set $ D \subset \supp \phi $ with $ \vert D \vert > \frac{\vert \supp \phi \vert}{2} $ on that for all $ x \in D  $ and $ \vert \gamma \vert \leq N  $ we have $ \vert D^{\gamma} \phi (x) \vert > C > 0   $. There is a set $ \tilde{D} \subset D  $ such that for all $ x \in \partial \tilde{D}  $ we have $ 2^{-10} > \dist( x , \partial D  ) > 2^{-20} $.  For $ k \in \mathbb{N}  $ we put $ n_{k} =  r ( k - 1 ) + 2  $ and $ x_{k} = 32^{r - 2} ( 1, 1, \ldots, 1)^{T} $. Moreover we put $  a_{k} = 2^{n_{k}( \frac{d}{u}- N)} $. We define the function
\begin{equation}\label{Os_func}
f(x) = \sum_{k = 1}^{\infty} a_{k} \phi ( 2^{n_{k}}x - x_{k}   ).
\end{equation}
Then for $ k \in \mathbb{N}  $ we have $  \supp  \phi ( 2^{n_{k}} \cdot - x_{k}   ) \subset B( 2^{- n_{k}} x_{k} ,   2^{- n_{k}}  )  \cap Q_{n_{k} , x_{k} }  $. Because of this we have $ \supp f \subset B( 0 , 4 \sqrt{d} \cdot 32^{r - 2}  +  4  )  $. For $ k, t \in \mathbb{N}  $ with $ k \not = t  $ we have
\begin{equation}\label{Os_sup1}
\supp \phi ( 2^{n_{k}} \cdot - x_{k}   ) \cap \supp \phi ( 2^{n_{t}} \cdot - x_{t}   ) = \emptyset .
\end{equation}
Moreover if we fix a large $ l \in \mathbb{N}  $ and $ h \in \R $ with $\vert h \vert \leq 2^{-n_{l + 1}} 2^{-rl}   $ for $ k, t \in \mathbb{N}  $ with $ 0 < k < t < l - 4  $ we have 
\begin{equation}\label{Os_sup2}
\supp \phi ( 2^{n_{k}} \cdot + N 2^{n_{k}} h  - x_{k}   ) \cap \supp \phi ( 2^{n_{t}} \cdot + N 2^{n_{t}} h - x_{t}   ) = \emptyset .
\end{equation}
Now we want to prove that we have $ \Vert f \vert \mathcal{N}^{N}_{u,p, \infty}(\R)  \Vert < \infty   $. For that purpose we want to use lemma \ref{lem_atom}. We already know that for $ k \in \mathbb{N} $ we have $ \supp  \phi ( 2^{n_{k}} \cdot - x_{k}   ) \subset  Q_{n_{k} , x_{k} }  $. For $ \vert \alpha \vert \leq  K \in \mathbb{N} $ with $ K > N + 1 $ because of $ \phi \in C_{0}^{\infty}(\R)  $ we have $ \Vert D^{\alpha}\phi ( 2^{n_{k}} \cdot - x_{k}   ) \vert L_{\infty}(\R) \Vert \leq C_{1} 2^{n_{k} \vert \alpha \vert } $. For $ \vert \beta \vert \leq L \in \mathbb{N}_{0} \cup \left\{ -1 \right\}   $ with $  L \geq \max( - 1 , \sigma_{p} - N )   $ we have 
\begin{align*}
\int_{\R} x^{\beta} \phi ( 2^{n_{k}} x - x_{k}   ) dx = 2^{- n_{k} \vert \beta \vert } 2^{-n_{k}d} \int_{\R} x^{\beta} \phi (  x - x_{k}   ) dx = 0  .
\end{align*}
So lemma \ref{lem_atom} can be applied and we obtain 
\begin{align*}
\Vert f \vert \mathcal{N}^{N}_{u,p, \infty}(\R)  \Vert & = \Big \Vert \sum_{k = 1}^{\infty} a_{k} \phi ( 2^{n_{k}}x - x_{k}   ) \Big \vert \mathcal{N}^{N}_{u,p, \infty}(\R) \Big  \Vert \\
& \leq C_{1} \sup_{k \in \mathbb{N} } 2^{n_{k}( N - \frac{d}{u}  )} \Big   \Vert  \vert a_{k} \vert \chi_{n_{k},x_{k}}^{(u)}  \Big \vert \mathcal{M}^{u}_{p}(\R) \Big  \Vert  \\
& = C_{1} \sup_{k \in \mathbb{N} }  \Big   \Vert  \chi_{n_{k},x_{k}}^{(u)}  \Big \vert \mathcal{M}^{u}_{p}(\R) \Big  \Vert .
\end{align*}
Now we use $ \Vert  \chi_{n_{k},x_{k}}^{(u)}  \vert \mathcal{M}^{u}_{p}(\R)   \Vert = 1  $, see the remark after definition 2.9 in \cite{HaMoSk}. We get
\begin{align*}
\Vert f \vert \mathcal{N}^{N}_{u,p, \infty}(\R)  \Vert \leq C_{1} < \infty .
\end{align*}
Next we will prove that we have $ \Vert f \vert \mathcal{N}^{N}_{u,p, \infty}(\R) \Vert^{( v , a )} = \infty $. To prove this at first we fix a large number $ l \in \mathbb{N} $ with $ l > 10  $. Then because of the disjoint supports of the involved functions we obtain
\begin{align*}
& \Big \Vert  \Big ( \int_{B(0,t)}\vert \Delta^{N}_{h}f(x) \vert^{v} dh \Big )^{\frac{1}{v}} \Big \vert  \mathcal{M}^{u}_{p}( \mathbb{R}^d) \Big \Vert^{p} \\
& \qquad \geq  \Big \Vert  \Big ( \int_{ \vert h \vert \leq \min( t , 2^{-n_{l + 1}} 2^{-rl} )} \Big ( \sum_{k = 1}^{ l - 6} a_{k} \vert \Delta^{N}_{h} ( \phi ( 2^{n_{k}} \cdot - x_{k}   ))(x) \vert \Big )^{v} dh \Big )^{\frac{1}{v}} \Big \vert  \mathcal{M}^{u}_{p}( \mathbb{R}^d) \Big \Vert^{p} .
\end{align*}
Now we use that for fixed $ h $ with $ \vert h \vert \leq \min( t , 2^{-n_{l + 1}} 2^{-rl} ) = t(r,l)  $ and $ k \in \mathbb{N} $ we have $ \supp  \Delta^{N}_{h} ( \phi ( 2^{n_{k}} \cdot - x_{k}   )) \subset   B( 0 ,  \sqrt{d} 32^{r - 2} ( 4 + N ) +  4  )    $. Because of this instead of the supremum of the Morrey quasi-norm we can choose the ball $ B( 0 ,  \sqrt{d} 32^{r - 2} ( 4 + N ) +  4  )    $. Then because of $ v \geq 1 $ we get
\begin{align*}
& \Big \Vert  \Big ( \int_{B(0,t)}\vert \Delta^{N}_{h}f(x) \vert^{v} dh \Big )^{\frac{1}{v}} \Big \vert  \mathcal{M}^{u}_{p}( \mathbb{R}^d) \Big \Vert^{p} \\
& \qquad  \geq C_{1}  \Big \Vert  \Big ( \int_{ \vert h \vert \leq t(r,l)} \Big ( \sum_{k = 1}^{ l - 6} a_{k} \vert \Delta^{N}_{h} ( \phi ( 2^{n_{k}} \cdot - x_{k}   ))(x) \vert \Big )^{v} dh \Big )^{\frac{1}{v}} \Big \vert  L_{p}( \mathbb{R}^d) \Big \Vert^{p} \\
& \qquad  \geq C_{2}  t(r,l)^{dp(\frac{1}{v}-1)} \Big \Vert   \int_{ \vert h \vert \leq t(r,l)}  \sum_{k = 1}^{ l - 6} a_{k} \vert \Delta^{N}_{h} ( \phi ( 2^{n_{k}} \cdot - x_{k}   ))(x) \vert  dh   \Big \vert  L_{p}( \mathbb{R}^d) \Big \Vert^{p} \\
& \qquad  = C_{2} t(r,l)^{dp(\frac{1}{v}-1)}  \sum_{k = 1}^{ l - 6} a_{k}^{p} \Big \Vert   \int_{ \vert h \vert \leq t(r,l)}  \vert \Delta^{N}_{2^{n_{k}} h} ( \phi (\cdot ))(2^{n_{k}}  x - x_{k} ) \vert dh  \Big \vert  L_{p}( \mathbb{R}^d) \Big \Vert^{p} \\
& \qquad  = C_{2} t(r,l)^{dp(\frac{1}{v}-1)}  \sum_{k = 1}^{ l - 6} a_{k}^{p} 2^{-n_{k} d}  \Big \Vert   \int_{ \vert h \vert \leq t(r,l)}   \vert \Delta^{N}_{2^{n_{k}} h} ( \phi (\cdot ))( x) \vert dh  \Big \vert  L_{p}( \mathbb{R}^d) \Big \Vert^{p}.
\end{align*}
Let $ \eta = \frac{h}{\vert h \vert} $ and $ \theta \in ( 0 , 1 )  $. Because of the properties of the sets $ D $ and $ \tilde{D}  $ we obtain 
\begin{align*}
& \Big \Vert  \Big ( \int_{B(0,t)}\vert \Delta^{N}_{h}f(x) \vert^{v} dh \Big )^{\frac{1}{v}} \Big \vert  \mathcal{M}^{u}_{p}( \mathbb{R}^d) \Big \Vert^{p} \\
& \;  \geq C_{2}  t(r,l)^{dp(\frac{1}{v}-1)} \sum_{k = 1}^{ l - 6} a_{k}^{p}  2^{-n_{k} d} 2^{n_{k}N p}  \Big \Vert  \int_{ \vert h \vert \leq \frac{ t(r,l)}{N}}  \Big \vert  \frac{\partial^{N}\phi}{\partial \eta^{N}} ( x + \theta N 2^{n_{k}} h ) \Big \vert  \vert h  \vert^{N} dh  \Big \vert  L_{p}( \tilde{D} ) \Big \Vert^{p} \\
& \;  \geq C_{3} t(r,l)^{dp(\frac{1}{v}-1)} \sum_{k = 1}^{ l - 6} a_{k}^{p}  2^{-n_{k} d} 2^{n_{k}N p}  \Big \Vert   \int_{ \vert h \vert \leq \frac{ t(r,l)}{N}}   \vert h  \vert^{N}  dh  \Big \vert  L_{p}( \tilde{D} ) \Big \Vert^{p} .
\end{align*}
When we use $  a_{k} = 2^{n_{k}( \frac{d}{u} - N)}  $ we get 
\begin{align*}
 \Big \Vert  \Big ( \int_{B(0,t)}\vert \Delta^{N}_{h}f(x) \vert^{v} dh \Big )^{\frac{1}{v}} \Big \vert  \mathcal{M}^{u}_{p}( \mathbb{R}^d) \Big \Vert^{p} &   \geq  C_{4} \sum_{k = 1}^{ l - 6} 2^{n_{k} p \frac{d}{u} } 2^{-n_{k}d}  t(r,l)^{N p  + d p + \frac{d p }{v} - d p}.
\end{align*}
Now because of $ r > 4  $ and $ l > 10 $ we have $  2^{-n_{l + 1}} 2^{-rl} < 1   $.  In the special case $ p = u $ this leads to
\begin{align*}
& \Big ( \Vert f \vert \mathcal{N}^{N}_{u,p, \infty}(\R) \Vert^{( v , a )} \Big )^{p} \\
& \qquad \geq C_{5} \sum_{k = 1}^{ l - 6}  2^{n_{k} p \frac{d}{u} } 2^{-n_{k}d}   \sup_{t \in [0,1]} t^{-Np - \frac{d p}{v}} \min( t ,  2^{-n_{l + 1}} 2^{-rl}  )^{N p  + d p + \frac{d p}{v} - d p}    \\
& \qquad \geq C_{5} \sum_{k = 1}^{ l - 6} 2^{n_{k} p \frac{d}{u} } 2^{-n_{k}d}      2^{( -n_{l + 1}- rl )(-N p - \frac{d p}{v})}    2^{(-n_{l + 1} - rl )(N p  + d p + \frac{d p}{v} - d p) }    \\
& \qquad = C_{5} \sum_{k = 1}^{ l - 6}  2^{n_{k} p \frac{d}{u} } 2^{-n_{k}d}     = C_{5} \sum_{k = 1}^{ l - 6} 1 = C_{5} ( l - 6 )  .
\end{align*}
So if $ l $ tends to infinity also $ \Vert f \vert \mathcal{N}^{N}_{u,p, \infty}(\R) \Vert^{( v , a )}  $ tends to infinity and the proof is complete. 
\end{proof}

\subsection{The proof of theorem \ref{MR2}}

Now we are able to prove theorem \ref{MR2}. For this purpose we have to sum up the results we obtained before.

{\bf \textit{Proof of theorem \ref{MR2}.} } To prove theorem \ref{MR2} we just have to combine the results from the propositions \ref{nec_0MR1}, \ref{nec_p1MR1}, \ref{nec_pvMR1} and \ref{nec_N>sMR1}. Then we obtain a result that is even stronger than theorem \ref{MR2}.

\section{Summary and outstanding issues}

In the theorems \ref{MR1} and \ref{thm_MR_va} we showed that under some sufficient conditions on the parameter $ s $ it is possible to describe the Besov-Morrey spaces in terms of differences. On the other hand in the propositions \ref{nec_0MR1}, \ref{nec_p1MR1}, \ref{nec_pvMR1} and \ref{nec_N>sMR1} we proved that in some cases such a characterization fails. When we compare both sides it turns out that there are still some open problems at the moment. So for $ s \in \mathbb{R}  $, $ 0 < p \leq u < \infty   $, $ 0 < q \leq \infty $, $ 0 < v \leq \infty $, $ 1 \leq a \leq \infty $ and $ N \in \mathbb{N} $ with $ N \geq s $ up to now we do not know whether we have $ \mathcal{N}^{s}_{u,p,q}(\R)  =  {\bf N}^{s, N, a}_{u, p, q, v}(\R)  $ if we are in one of the following situations. 

\begin{itemize}

\item[(i)]

We have $ s = 0 $ with $ 1 \leq a < \infty  $ and either $ p \geq 2  $ and $  q \leq 2 $ or $ 1 \leq p < 2  $ and $ q \leq p   $.

\item[(ii)]

We have $   d \frac{p}{u}  (   \frac{1}{p} - 1  ) < s \leq d  (   \frac{1}{p} - 1  ) $ with $ 0 < p < 1 $ and $ 0 < v < 1 $.

\item[(iii)]

We have $  d  \frac{p}{u}  (   \frac{1}{p} - \frac{1}{v}  ) \leq s \leq d    (   \frac{1}{p} - \frac{1}{v}  ) $ with $ v > \max(1,p) $.

\item[(iv)]

We have $ N = s  $ with $ q = \infty  $ and $ p < u  $.

\end{itemize}

\large{

{\bf Acknowledgements }

The author is funded by a Landesgraduiertenstipendium which is a scholarship from the Friedrich-Schiller university and the Free State of Thuringia. The author would like to thank his supervisor professor Winfried Sickel for his tips and hints. Moreover he would like to thank professor Dorothee D. Haroske for the nice diagram that can be found in the introduction of this paper. 

}


\end{document}